\documentclass[a4paper,12pt]{amsart}

\usepackage{verbatim}                          %for use of \begin{comment} ... \end{comment}
\usepackage{amsmath,amsthm}                                        %always load amsmath before amsthm!
\usepackage{color,textcomp,amssymb,german}
\usepackage{graphicx}

\DeclareFontEncoding{OT2}{}{} % to enable usage of cyrillic fonts
\newcommand{\textcyr}[1]{{\fontencoding{OT2}\fontfamily{wncyr}\fontseries{m}\fontshape{n}\selectfont #1}}
\newcommand{\Sha}{{\mbox{\textcyr{Sh}}}}

\usepackage[british]{babel}
\usepackage[all,import]{xy}%\CompileMatrices\CompilePrefix{xymatrix/diagram} %xypic, diagrams are saved as xymatrix/diagram04.xyc
\usepackage[T1]{fontenc}                                    %for better hyphenation
\usepackage{times}                                       
\usepackage[mathscr]{eucal}
\swapnumbers
\newtheoremstyle{theorem}{8pt\vfill}{8pt\vfill}{\itshape}{}{\bfseries}{.}{.5em}{}
\newtheoremstyle{paragraph}{8pt\vfill}{8pt\vfill}{}{}{\bfseries}{}{.5em}{}
\theoremstyle{theorem}                                       %theorems like: 5.12 Theorem. Blabla...
\newtheorem{thm}{Theorem}[section]
\newtheorem{prop}[thm]{Proposition}
\newtheorem{cor}[thm]{Corollary}
\newtheorem{lemma}[thm]{Lemma}
\theoremstyle{definition}

\newtheorem{eg}[thm]{Example}
\newtheorem{rem}[thm]{Remark}

\theoremstyle{paragraph}                                     %paragraph number: 5.13. Blabla...
\newtheorem{pg}[thm]{\!\!}
\newcommand{\norm}[1]{\left| #1 \right|}
\newcommand{\lrquot}[3]{#1 \,\backslash\, #2 \,/\, #3 }
\newcommand{\lquot}[2]{#1  \,\backslash\, #2 }
\newcommand{\rquot}[2]{#1  \, / \, #2 }

\newdir{ >}{{}*!/-5pt/@^{(}}\newcommand{\arincl}[1]{\ar@{ >->}@<-0,0ex>#1}    %inclusions arrows with better spacing

\def \smallmat #1 #2 #3 #4 {{\scriptstyle \begin{pmatrix} {#1} & {#2} \\ {#3} & {#4} \end{pmatrix}}}
\def \tinymat #1 #2 #3 #4 {\bigl( \begin{smallmatrix} {#1} & {#2} \\ {#3} & {#4} \end{smallmatrix} \bigr)}
\def\eqref #1{\textup{(}\ref{#1}\textup{)}}

\newcommand{\ses}[3]{\xymatrix{0\ar[r]&{#1}\ar[r]&{#2}\ar[r]&{#3}\ar[r]&0}}                %short exact sequence

\def\blanc{\mspace{4mu}\cdot\mspace{4mu}}
\def\cT{\mathcal{T}}
\def\cO{\mathcal{O}}

\def\cV{\mathcal{V}}

\def\cA{\mathcal{A}}
\def\cH{\mathcal{H}}
\def\cF{\mathcal{F}}
\def\cL{\mathcal{L}}
\def\cI{\mathcal{J}}
\def\cM{\mathcal{M}}
\def\cK{\mathcal{K}}
\def\cN{\mathcal{N}}
\def\cX{\mathcal{X}}
\def\cU{\mathcal{U}}
\def\cN{\mathcal{N}}
\def\cE{\mathcal{E}}
\def\cG{\mathcal{G}}
\def\cC{\mathcal{C}}

\def\cB{\mathcal{B}}
\def\cE{\mathcal{E}}
\def\cR{\mathcal{R}}

\def\ZZ{\mathbb{Z}}
\def\QQ{\mathbb{Q}}

\def\CC{\mathbb{C}}
\def\FF{\mathbb{F}}

\def\AA{\mathbb{A}}
\def\PP{\mathbb{P}}

\DeclareMathOperator{\Hom}{Hom}

\DeclareMathOperator{\Ext}{Ext}

\DeclareMathOperator{\Stab}{Stab}

\DeclareMathOperator{\Gal}{Gal}

\DeclareMathOperator{\tr}{tr}

\DeclareMathOperator{\PGL}{PGL}
\DeclareMathOperator{\SL}{SL}
\DeclareMathOperator{\Bun}{Bun}
\DeclareMathOperator{\Pic}{Pic}
\DeclareMathOperator{\Cl}{Cl}

\DeclareMathOperator{\Vertex}{Vert\,}
\DeclareMathOperator{\Edge}{Edge\,}

\def\PBun{\PP\!\Bun}
\def\PBundec{\PP\!\Bun_2^{\rm dec}}

\def\PBuntr{\PP\!\Bun_2^{\rm tr}}
\def\PBungi{\PP\!\Bun_2^{\rm gi}}

\def\tor{{\rm tor}}

\def\i{{\rm i}}

\def\nsimeq{\simeq\hspace{-10.5pt}/\hspace{5.5pt}}
\def\nequiv{\equiv\hspace{-11pt}/\hspace{5.5pt}}
\def\longhookrightarrow{\lhook\joinrel\longrightarrow}

\renewcommand{\date}[1]{\gdef\DD{#1}}\newcommand{\DD}{}

\title{Automorphic forms for elliptic function fields}
\author{Oliver Lorscheid}
\address{The City College of New York, Math. Dept., 160 Convent Ave., New York NY 10031, USA}
\email{olorscheid@ccny.cuny.edu}

\begin{document}

\begin{abstract}
 Let $F$ be the function field of an elliptic curve $X$ over $\FF_q$. In this paper, we calculate explicit formulas for unramified Hecke operators acting on automorphic forms over $F$.  We determine these formulas in the language of the graph of an Hecke operator, for which we use its interpretation in terms of $\PP^1$-bundles on $X$. This allows a purely geometric approach, which involves, amongst others, a classification of the $\PP^1$-bundles on $X$.

 We apply the computed formulas to calculate the dimension of the space of unramified cusp forms and the support of a cusp form. We show that a cuspidal Hecke eigenform does not vanish in the trivial $\PP^1$-bundle. Further, we determine the space of unramified $F'$-toroidal automorphic forms where $F'$ is the quadratic constant field extension of $F$. It does not contain non-trivial cusp forms. An investigation of zeros of certain Hecke $L$-series leads to the conclusion that the space of unramified toroidal automorphic forms is spanned by the Eisenstein series $E(\blanc,s)$ where $s+1/2$ is a zero of the zeta function of $X$---with one possible exception in the case that $q$ is even and the class number $h$ equals $q+1$.
\end{abstract}

%{\ \\ \flushright\tiny version 0.5\\ \vspace{-0pt}\today\\ }\vspace{-32pt}

\maketitle

\tableofcontents

%%%%%%%%%%%%%%%%%%%%%%%%%%%%%%%%%%%%%%%%%%%%%%%%%%%%%%%%%%%%%%%%%%%%%%%%%%%%%%%%%%%%%%%%%%%%%%%%%%%%%%%%%%%%%%%%%%%%%%%%%%%%%%%%%%%%%%%%%%%%%%%%%%%%%%%%%%%%%%%%%%%%%%%%%%%%%%%%%%%%%%%%%%

\section*{Introduction}

\noindent
The space of automorphic forms over $\QQ$ is an extensively studied object and many explicit results about its structure are known, for instance, one knows the dimensions of certain subspaces like spaces of cusp forms with a fixed weight and ramification. The function field analog of the rational numbers are rational function fields $\FF_q[T]$. The space of automorphic forms is also in these cases well-studied. A major tool of investigation are Hecke operators, for which explicit formulas are available for both $\QQ$ and $\FF_q[T]$. The action of Hecke operators is less known for other global fields than $\QQ$ and $\FF_q[T]$. In particular, explicit formulas for the action of Hecke operators over function fields of genus $1$ or higher are not available in literature---with one exception, which are elliptic function fields with odd class number. We explain below how to extract these explicit formulas from results in literature and why this method works only for odd class number. We will see that the only satisfying cases are elliptic function fields with class number $1$, which includes, up to isomorphism, only three fields (cf.\ Example \ref{eg_class_number_1}).

In this paper we shall work out explicit equations for Hecke operators over any elliptic function field. This will be done in terms of the graph of an Hecke operator, which is a tool introduced exactly for this purpose (cf.\ \cite{Lorscheid2}). We extract formulas from these graphs and employ them to calculate the space of cusp forms, the space of toroidal automorphic forms as well as some other spaces. Note that in this exposition as well as throughout the paper, we will restrict ourselves to unramified automorphic forms and unramified Hecke operators exclusively, and we agree to suppress the attribute ``unramified'' from our terminology.

We continue with explaining what can be deduced from literature about Hecke operators for elliptic function fields. Let the $F$ be the function field of an elliptic curve $X$ over a finite field $\FF_q$. Let $x$ be a place of $F$. We denote by $F_x$ the completion of $F$ at $x$, by $\cO_x$ its integers, by $\pi_x\in\cO_x$ a uniformizer and by $q_x$ the cardinality of the residue field $\rquot{\cO_x}{(\pi_x)}\simeq\FF_{q_x}$. The Bruhat-Tits tree $\cT_x$ of $F_x$ is a graph with vertex set $\rquot{\PGL_2(F_x)}{\PGL_2(\cO_x)}$. There is an edge between two cosets $[g]$ and $[g']$ if and only if $[g']$ contains $g\tinymat 1 {} {} {\pi_x} $ or $g\tinymat {\pi_x} b {} 1 $ for some $b\in\FF_{q_x}$. Note that this condition is symmetric in $g$ and $g'$, so $\cT_x$ is a geometric graph. In fact, $\cT_x$ is a $(q_x+1)$-regular tree.

We define the action of the local Hecke operator $T_x$ on the space of complex valued functions $f$ on the vertices of $\cT_x$ by the formula
$$ T_x(f)(v) \quad = \quad \sum_{v'\text{ adjacent to }v} f(v')\;. $$
Let $\cO^x_F\subset F$ be the Dedekind ring of all elements $a\in F$ with $\norm a_y\leq 1$ for all places $y\neq x$. Then $\Gamma=\PGL_2(\cO_F^x)$ acts on $\cT_x$ by left multiplication, which induces an action of $\Gamma$ on the functions on $\Vertex\cT_x$. This action commutes with the action of $T_x$.

If the class number $h$ of $F$ as well as the degree of $x$ is odd, then the strong approximation property of $\SL_2$ implies that the inclusion of $\PGL_2(F_x)$ into $\PGL_2(\AA)$ induces a bijection
$$ \lquot{\Gamma}{\Vertex\cT_x} \quad\stackrel\sim\longrightarrow\quad \lrquot{\PGL_2(F)}{\PGL_2(\AA)}{\PGL_2(\cO_\AA)} $$
where $\cO_\AA$ is the maximal compact subring of the adeles $\AA$ of $F$ (cf.\ \cite[Prop.\ 3.8]{Lorscheid2}. We denote the quotient on the right hand side by $\cX$. An automorphic form is a function $f$ on $\cX$, or, equivalently, on $\lquot{\Gamma}{\Vertex\cT_x}$, which satisfies, as a ($\Gamma$-invariant) function on $\Vertex\cT_x$, that $\{T_x^i(f)\}_{i\geq0}$ spans a finite-dimensional complex vector space. Note that the local Hecke operator $T_x$ corresponds to a (global) Hecke operator $\Phi_x$. To be more precise, the bijection $\lquot{\Gamma}{\Vertex\cT_x}\to\cX$ induces an isomorphism between the function spaces on these sets, which is equivariant with respect to the operators $T_x$ and $\Phi_x$.

Thus we can regard $f$ as a function on the quotient $\lquot{\Gamma}{\Vertex\cT_x}$. Shuzo Takahashi calculates this quotient for places $x$ of degree $1$ in \cite{Takahashi}. This means that he describes representatives in $\PGL_2(F_x)$ for the double quotient $\lrquot{\Gamma}{\PGL_2(F_x)}{\PGL_2(\cO_x)}$. The full subgraph of the classes of these representatives in $\Vertex\cT_x$ is a tree, and this tree is isomorphic to the quotient graph $\lquot{\Gamma}{\cT_x}$. 

We illustrate the quotient graph together with some matrix representatives of vertices in Figure \ref{figure_takahashi}. The element $b$ varies through all elements in $\FF_q$. Whether $\tinymat {\pi_x^2} b {} 1 $ represents a vertex to the right or to the left of $\tinymat \pi_x 1 {} 1 $ depends on whether the Weierstrass polynomial $P(\underline X,\underline Y)$ for $X$ has a root in $\underline X$ for $\underline Y=b$ or not.

\begin{figure}[htb]
 \begin{center}
  \includegraphics{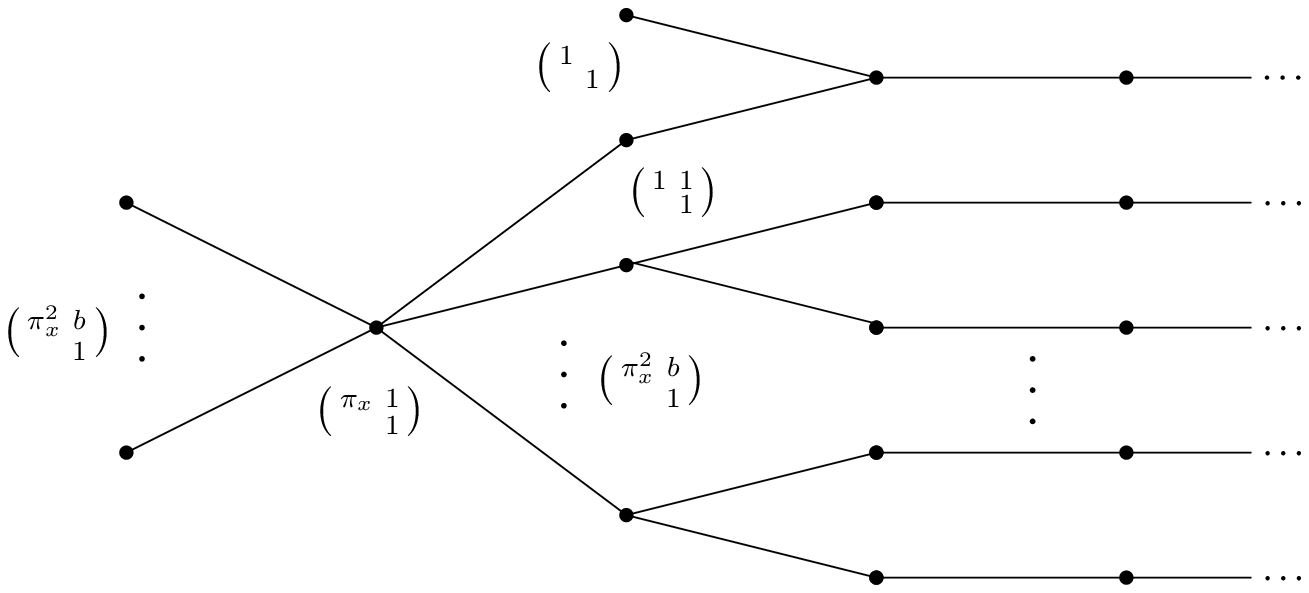} 
  \caption{The quotient graph $\lquot{\Gamma}{\cT_x}$}
  \label{figure_takahashi}
 \end{center} 
\end{figure}

This quotient graph, in turn, is isomorphic to the graph of $\Phi_x$ (when weights are suppressed), which is illustrated in Figure \ref{ell_h_odd_graph} (see section \ref{section_examples}). Takahashi further determines the stabilizers of $\Gamma$ acting $\Vertex\cT_x$, which allows to compute the action of $T_x$ on a function on $\lquot{\Gamma}{\Vertex\cT_x}\simeq\cX$ by the formula
$$ T_x(f)(v) \quad = \quad \sum_{\substack{\text{edges }e\text{ with origin }v\\\text{and terminus }v'}} [\Stab_\Gamma(v):\Stab_\Gamma(e)] \cdot f(v')\;. $$

An automorphic form that is an simultaneous eigenfunction for all Hecke operators is determined by the set of its eigenvalues. In the case that the class number is $1$, the automorphic form is already determined by the eigenvalue with respect to $\Phi_x$ respective $T_x$, up to finitely many exceptions. For computations that profit from this point of view, see \cite{Cornelissen-Lorscheid}. If the class number is odd, but not $1$, then the eigenvalue with respect to $T_x$ does not suffice to determine an automorphic form, but we need to consider the action of other Hecke operators as well. This leads to the difficulty to identify the representatives of $\cX$ in $\PGL_2(F_x)$ with the representatives in $\PGL_2(F_y)$ where $y$ is a different place of degree $1$. From the viewpoint of Takahashi's paper, this seems to be a difficult problem. Even worse, if the class number is even, the correspondence between the local and the global situation breaks down and we are not able to draw any of the above conclusions.

For this reason, the notion of the graph of an Hecke operator was introduced in \cite{Lorscheid2}, which relates to automorphic forms as functions on $\cX$ directly, without making use of the tree $\cT_x$. The interpretation of $\cX$ as the set of isomorphism classes of $\PP^1$-bundles on $X$ allows us to apply geometric methods, which prove to be very efficient. This gives access to a simultaneous consideration of all Hecke operators, for both odd and even class number. Note that we will restrict ourselves to the geometric viewpoint for the rest of this paper, in contrast to the above exposition. More details on the relation between the geometric and the arithmetic setting can be found in \cite[section 5]{Lorscheid2}.

The paper is organized into two parts. In part \ref{part_graphs}, we determine the graphs $\cG_x$ of the Hecke operators $\Phi_x$ for degree $1$ places $x$. In section \ref{section_graphs}, we recall the definition of the graphs $\cG_x$ in the more general setting of an arbitrary global function field. We review the structure theory for $\cG_x$ as developed in \cite{Lorscheid2}. In particular, we explain that $\cG_x$ decomposes into finitely many cusps, which are subgraphs that have a simple description, and into a finite subgraph, which is called the nucleus of $\cG_x$. In section \ref{section_vertices}, we determine the vertices of the nucleus as a consequence of Atiyah's classification of vector bundles on elliptic curves. In section \ref{section_edges}, we determine the edges in the nucleus by extensive use of the methods from \cite{Lorscheid2}. In section \ref{section_examples}, we illustrate examples of graphs of Hecke operators. 

In part \ref{part_applications}, we apply the knowledge about the graphs $\cG_x$ to explicit calculations with automorphic forms. In section \ref{section_automorphic}, we review the notion of an automorphic form as a function on the graph. We write out explicit formulas for an Hecke operator $\Phi_x$ acting on an eigenfunction. In section \ref{section_cusp}, we calculate the space of cusp forms. This means that we determine its dimension and the (maximal) support of a cusp form. We show that the space of cusp forms is a $0$-eigenspace for every $\Phi_x$ where $x$ is a place of odd degree. We prove further that a cuspidal Hecke eigenform does not vanish in the trivial $\PP^1$-bundle. In section \ref{section_toroidal}, we calculate the space of $F'$-toroidal automorphic forms where $F'$ is the quadratic constant field extension of $F$. In particular, we see that this space contains co non-trivial cusp form. We show that the space of toroidal automorphic forms is generated by a single Eisenstein series $E(\blanc,s)$ where $s+1/2$ is a zero of the zeta function of $F$---with a possible exception in the case that the characteristic of $F$ is $2$ and the class number $h$ of $F$ equals $q+1$ where the space of toroidal automorphic forms might be $2$-dimensional.

\medskip
{\bf Acknowledgements:} This paper is extracted from my thesis \cite{Lorscheid-thesis}. I would like to thank Gunther Cornelissen for his advice during my graduate studies.

\bigskip

%%%%%%%%%%%%%%%%%%%%%%%%%%%%%%%%%%%%%%%%%%%%%%%%%%%%%%%%%%%%%%%%%%%%%%%%%%%%%%%%%%%%%%%%%%%%%%%%%%%%%
%%%%%%%%%%%%%%%%%%%%%%%%%%%%%%%%%%%%%%%%%%%%%%%%%%%%%%%%%%%%%%%%%%%%%%%%%%%%%%%%%%%%%%%%%%%%%%%%%%%%%

\part{Graphs of Hecke operators for elliptic function fields}
\label{part_graphs}

\bigskip

%%%%%%%%%%%%%%%%%%%%%%%%%%%%%%%%%%%%%%%%%%%%%%%%%%%%%%%%%%%%%%%%%%%%%%%%%%%%%%%%%%%%%%%%%%%%%%%%%%%%%

\section{Reminder on graphs of Hecke operators}
\label{section_graphs}

\noindent
In this section, we recall the definition of the graphs $\cG_x$ of the Hecke operators $\Phi_x$ as introduced in \cite{Lorscheid2}. These graphs encode the action of certain unramified Hecke operators $\Phi_x$ which act on the space of automorphic forms for $\PGL_2$ over a global function field $F$. The connection to the Hecke operators $\Phi_x$ will be explained in section \ref{section_automorphic}. We concentrate in this resume on the geometric point of view. For the translation into adelic language, see section 5 in \cite{Lorscheid2}. 

\begin{pg}\label{def_hecke_graphs}
 Let $q$ be a prime power and $X$ a smooth projective geometrically irreducible curve over $\FF_q$ with function field $F$. Let $\norm X$ be the set of closed points of $X$, which we identify with the set of places of $F$. 
 
 Let $X'=X\otimes\FF_{q^2}$ be the constant field extension of $X$ to $\FF_{q^2}$ and $\overline X=X\otimes\overline\FF_q$ the constant field extension of $X$ to the algebraic closure $\overline\FF_q$ of $\FF_q$. For $Y$ equal to $X$, $X'$ or $\overline X$, we denote by $\Pic Y$ the Picard group of  $Y$ and by $\Bun_2 Y$ the set of isomorphism classes of rank $2$ bundles over $Y$. The Picard group $\Pic Y$ acts on $\Bun_2 Y$ via the tensor product. We denote the quotient by $\PBun_2 Y$, which is the same as the set of isomorphism classes of $\PP^1$-bundles over $Y$. We write $[\cM]\in\PBun_2 Y$ if the class $[\cM]$ is represented by the rank $2$ bundle $\cM$, and $\cM\sim\cM'$ if $[\cM]=[\cM']$. We identify $\cM$ with the associated locally free sheaf of rank $2$. 

 Two exact sequences of sheaves
 $$ 0\to\cF_1\to\cF\to\cF'_1\to0 \hspace{1cm} \textrm{and} \hspace{1cm} 0\to\cF_2\to\cF\to\cF'_2\to0\;, $$
 are \emph{isomorphic with fixed $\cF$} if there are isomorphisms $\cF_1\to\cF_2$ and $\cF'_1\to\cF'_2$ such that
 $$ \xymatrix{0\ar[r]&{\cF_1}\ar[r]\ar[d]^\simeq&{\cF}\ar[r]\ar@{=}[d]&{\cF'_1}\ar[r]\ar[d]^\simeq&0\\ 0\ar[r]&{\cF_2}\ar[r]&{\cF}\ar[r]&{\cF'_2}\ar[r]&0} $$
 commutes.

 Fix a place $x$. Let $\cK_x$ be the torsion sheaf that is supported at $x$ and has stalk $\kappa_x$ at $x$, where $\kappa_x$ is the residue field at $x$. Fix a representative $\cM$ of $[\cM]\in\PBun_2 X$. We define $m_x([\cM],[\cM'])$ as the number of isomorphism classes of exact sequences
 $$ \ses{\cM''}{\cM}{\cK_x} $$
 with fixed $\cM$ and with $\cM''\sim\cM'$. This number is independent of the choice of representative $\cM$ (cf.\ \cite[par.\ 5.4]{Lorscheid2}).

 For a $\PP^1$-bundle $v\in\PBun_2X$ we define
 $$ \cU_x(v) \ = \ \left\{ (v,v',m) \ \left|\ m=m_x(v,v')\neq0 \right.\right\} \;, $$
 and call the occurring $v'$ the {\it $\Phi_x$-neighbours of $v$}, and $m_x(v,v')$ their {\it multiplicity}.

 We define the graph $\cG_x$ by 
 \begin{align*} 
   \Vertex \cG_x \ &= \ \PBun_2 X\quad \textrm{ and} \\
   \Edge \cG_x \ &= \ \coprod_{v\in\PBun_2X} \cU_x(v) 
 \end{align*}
 where an edges from $v$ to $v'$ comes together with a weight $m=m_x(v,v')$.

 We list some facts about the graphs $\cG_x$. By definition, the weight of an edge is a positive integer, and there is at most one edge between two vertices. Every edge $(v,v',m)$ has an inverse edge, i.e.\ there is an edge $(v',v,m')$ in $\cG_x$ for some positive integer $m'$, which in general differs from $m$ (cf.\ \cite[par.\ 3.2]{Lorscheid2}). The weights of all edges $(v,v',m)$ with origin $v$ sum up to $q_x+1$ where $q_x=q^{\deg x}$ is the cardinality of $\kappa_x$ (cf.\ \cite[Prop. 2.3]{Lorscheid2}). In particular, $\cG_x$ is a locally finite graph. We illustrate a single edge $(v,v',m)$ by
 $$ \includegraphics{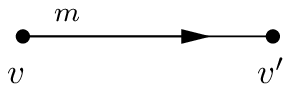} $$
 and a pair of inverse edges $(v,v',m)$ and $(v',v,m')$ by 
 $$ \includegraphics{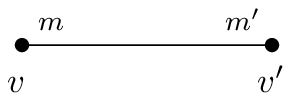} $$
 (note that in all examples of the present paper $v\neq v'$ for an edge $(v,v',m)$ of $\cG_x$). We encourage the reader to have a glance at the examples in section \ref{section_examples}.
\end{pg}

\begin{pg} \label{char_PBundec}\label{char_PBuntr}
 More details on the facts listed in this paragraph can be found in sections 6 and 7 of \cite{Lorscheid2}. A line subbundle of a rank $2$ bundle $\cM$ is a morphism $\cL\to\cM$ where $\cL$ is a line bundle such that the quotient $\rquot{\cM}{\cL}$ is torsion free and thus also a line bundle. We define $\delta(\cL,\cM)$ as $2\deg\cL-\deg\cM$ and $\delta(\cM)$ as the supremum of $\delta(\cL,\cM)$ over all subbundles $\cL\to\cM$. If $\cM\sim\cM'$, then $\delta(\cM)=\delta(\cM')$, so $\delta([\cM])=\delta(\cM)$ is well-defined for $[\cM]\in\PBun_2 X$. For all rank $2$ bundles $\cM$ there is a line subbundle $\cL\to\cM$ such that $\delta(\cM)=\delta(\cL,\cM)$. We call such a line subbundle a \emph{maximal subbundle of $\cM$}. For every $v\in\Vertex\cG_x=\PBun_2 X$, we have $\delta(v)\geq -2g$ where $g$ is the genus of $X$.

 The set $\PBun_2 X$ is the disjoint union of the following three subsets: the set $\PBundec X$ of classes that are represented by rank $2$ bundles that decompose into a direct sum of two line bundles, the set $\PBuntr X$ of classes that are represented by indecomposable rank $2$ bundles that decompose over $X'$ into the sum of two line bundles and the set $\PBungi X$ of classes that are represented by geometrically indecomposable rank $2$ bundles. 

 Classes in $\PBundec X$ can be represented by a rank $2$ bundle of the form $\cO_X\oplus\cL_D$ where $\cO_X$ is the structure sheaf of $X$ and $\cL_D$ is the line bundle associated to the divisor class $D$ in the divisor class group $\Cl X$ (cf.\ \cite[Prop. II.6.13]{Hartshorne}). We denote the corresponding class in $\PBundec X$ by $c_D$. We have that $c_D=c_{D'}$ if and only if $D= D'$ or $D=-D'$ in $\Cl X$ (cf.\ \cite[Prop.\ 6.3]{Lorscheid2}). 

 Classes in $\PBuntr X$ are represented by the trace of a line bundle $\cL$ over $X'$. A line bundle over $X'$ corresponds to a divisor class $D\in\Cl X'$, and we denote the class in $\PBun_2 X$ represented by the trace of $\cL_D$ by $t_D$. Note that $t_D=c_0$ if and only if $D\in\Cl X\subset\Cl X'$, and otherwise $t_D\in\PBuntr X$. We have $t_D=t_{D'}$ if and only if $D= D'$ or $D=-D'$ in $\rquot{\Cl X'}{\Cl X}$ (cf.\ \cite[Prop.\ 6.4]{Lorscheid2}). The integer $\delta(v)$ is even and negative for $v\in\PBuntr X$. 

 The set $\PBungi X$ depends heavier on the arithmetic of the given curve $X$. For $v\in\PBungi X$, we have that $-2g\leq\delta(v)\leq2g-2$. Consequently, $\PBungi X$ is empty if the genus of $X$ is $0$. For genus $1$, we will determine $\PBungi X$ in Theorem \ref{ell_PBungi}.
\end{pg}

\begin{pg}
 We recall the definitions of the nucleus and the cusps of $\cG_x$. More details on the  following can be found in section 8 of \cite{Lorscheid2}. Let $m_X=\max\{0,2g-2\}$. The \emph{nucleus $\cN_x$ of $\cG_x$} is the full subgraph of $\cG_x$ whose vertex set consists of all those $v\in\Vertex\cG_x$ with $\delta(v)\leq m_X+\deg x$. For every $D\in\Cl X$, the \emph{cusp $\cC_x(D)$} is defined as the full subgraph of $\cG_x$ whose vertex set consists of all vertices of the form $c_{D'}$ with $D'\equiv D\pmod{\langle x\rangle}$ and $\deg D'>m_X$. In particular, a cusp depends only on the class $[D]\in\Cl X/\langle x\rangle$. These classes are represented by $D\in\Cl X$ with $m_X<\deg D\leq m_X+\deg x$; consequently there are $h\deg x$ cusps where $h=\#\Cl^0 X$ is the class number of $F$. If $D$ is such a representative, then the cusp $\cC_x(D)$ looks like
 $$ \includegraphics{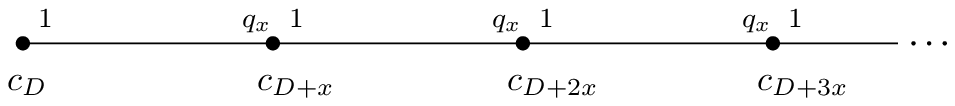} $$
 The graph $\cG_x$ is the union of the nucleus and the cusps. The union of edges is disjoint. The vertex sets of different cusps are disjoint and the intersection of $\Vertex\cN_x$ with $\Vertex\cC_x(D)$ equals $\{c_D\}$ if $m_X<\deg D\leq m_X+\deg x$. Note that both $\PBuntr X$ and $\PBungi X$ are contained in the vertex set of $\cN_x$. 

 The connected components of $\cG_x$ stay in bijection with the $2$-torsion elements of $\Pic X$. In particular, $\cG_x$ is connected if and only if the class number of $F$ is odd.
\end{pg}

%%%%%%%%%%%%%%%%%%%%%%%%%%%%%%%%%%%%%%%%%%%%%%%%%%%%%%%%%%%%%%%%%%%%%%%%%%%%%%%%%%%%%%%%%%%%%%%%%%%%%

\section{Vertices}
\label{section_vertices}

\noindent
 In this section, we determine the set $\PBun_2 X$ of all isomorphism classes of $\PP^1$-bundles for a curve $X$ of genus $1$. In paragraph \ref{char_PBundec}, we already described the subsets $\PBundec X$ and $\PBuntr X$. This reduces the problem to classifying the classes in $\PBungi X$.

\begin{pg}
 \label{ell_basics}
  From now on, let $X$ be a curve of genus $1$ over $\FF_q$ with function field $F$, $\Cl X$ the divisor class group and $h$ the class number. The canonical sheaf $\omega_X$ is isomorphic to the structure sheaf $\cO_X$. The map $X(\FF_q)\to\Cl^1 X$, which is defined by considering an $\FF_q$-rational point as a prime divisor of degree $1$, is a bijection. We identify these sets. The choice of an $x_0\in X(\FF_q)$ defines the bijection
 $$ \begin{array}{ccc}
     X(\FF_q)&\stackrel\sim\longrightarrow&\Cl^0 X \;.\\
        x    &\longmapsto    & x-x_0
    \end{array} $$	  
 So $X(\FF_q)$ inherits a group structure and $X$ becomes an elliptic curve. %For this reason, a global function field of genus $1$ is also called an elliptic function field\index{elliptic function field}.
\end{pg}

\begin{pg} 
 \label{rr_for_ell_curves}
 The Riemann-Roch theorem reduces to $\dim_{\FF_q}\!\Gamma(\cL)-\dim_{\FF_q}\!\Gamma(\cL^{-1})=\deg \cL$. Since $\Gamma(\cL)$ is non-zero if and only if $\cL$ is associated to an effective divisor (cf.\ \cite[Prop. II.7.7(a)]{Hartshorne}), we obtain:
 $$ \dim_{\FF_q}\!\Gamma(\cL) \ = \ \left\{ \begin{array}{ll} 0&\text{if }\deg\cL\leq0\text{ and }\cL\nsimeq\cO_X,\\
                                                     1&\text{if }\cL\simeq\cO_X,\\
							 \deg\cL&\text{if }\deg\cL>0.  \end{array} \right. $$
 
 By Serre duality, $\Ext^1(\cO_X,\cO_X)\simeq \Hom(\cO_X,\cO_X)\simeq\Gamma(\cO_X)$ is one-dimensional. The multiplicative group $\FF_q^\times$ acts on $\Ext^1(\cO_X,\cO_X)$, and this action preserves the isomorphism type of the rank $2$ bundle $\cM$ that is determined by an extension of $\cO_X$ by itself (cf.\ \cite[par.\ 7.3]{Lorscheid2}). Consequently, there is only one rank $2$ bundle $\cM_0$, up to isomorphism, that is a non-trivial extension of $\cO_X$ by itself. Since $\delta(\cO_X,\cM_0)=0$, \cite[Lemma 7.6]{Lorscheid2} implies that $\delta(\cM_0)=0$, and since $\cM_0\nsimeq\cO_X\oplus\cO_X$, the vector bundle $\cM_0$ is indecomposable. Since for $v\in\PBuntr X$, we have $\delta(v)<0$, it follows that $[\cM_0]\in\PBungi X$. We denote this class by $s_0$.

 Let $x$ be a place of degree $1$ and let $\cL_x$ denote the line bundle associated to the divisor class $[x]\in\Cl X$. The $\FF_q$-vector space $\Ext^1(\cO_X,\cL_x)\simeq\Hom(\cO_X,\cL_x)\simeq\Gamma(\cL_x)$ is also one-dimensional, and the non-trivial extensions define a rank $2$ bundle $\cM_x$. In this case, $\delta(\cM_x)=\delta(\cO_X,\cM_x)=\deg\cO_X-\deg\cL_x=-1$ and $\cO_X\to\cM_x$ is a maximal subbundle. Indeed, if there was a subbundle $\cL\to\cM_x$ of positive degree, \cite[Lemma 7.6]{Lorscheid2} would imply that $\deg\cL=1$ and $\cM_x\simeq\cO_X\oplus\cL$. But $\cL\simeq\det(\cO_X\oplus\cL)\simeq\det\cM_x\simeq\cL_x$, thus we contradict the assumption that $\cM_x$ is a non-trivial extension of $\cL_x$ by $\cO_X$. By the considerations of paragraph \ref{char_PBundec} on the values of $\delta$ we know that $[\cM_x]\in\PBungi X$. We denote this class by $s_x$.
\end{pg}

\begin{rem}
 \label{rem_D_vs_[D]}
 The notations for the vector bundles $\cL_x$ and $\cM_x$ of the previous paragraph is the same as the notation for the stalk of some vector bundles $\cL$ respective $\cM$ at $x$. To avoid confusion, we will reserve  the notations $\cL_x$ and $\cM_x$ strictly for the vector bundles as defined in the last paragraph throughout the whole paper.
\end{rem}

\begin{thm}
 \label{ell_PBungi}
 $$\PBungi X \ = \ \bigl\{\ s_x \ \bigl| \ x\in\Cl^1 X\ \bigr.\bigr\} \ \amalg \ \bigl\{ \ s_0 \ \bigr\} \;, $$
 and $\ s_x = s_y\ $ if and only if $\ (x-y) \in 2\Cl^0 X$. 
\end{thm}

\begin{proof}
 Let $Y$ denote one of $X$, $X'= X \otimes \overline\FF_{q^2}$ or $\overline X = X \otimes \overline\FF_q$. Let $\cB_n^d(Y)$ be the set of isomorphism classes of geometrically indecomposable rank $n$ bundles over $Y$ that have degree $d$.  We have inclusions $\cB_n^d(X)\subset\cB_n^d(X')\subset\cB_n^d(\overline X)$ (cf.\ \cite[par.\ 6.1]{Lorscheid2}). For a rank $1$ bundle $\cL$ over $Y$, the map
 $$ \begin{array}{ccc}\cB_n^d(Y)&\longrightarrow&\cB_n^{d+rn}(Y)\\
                      \cM       &\longmapsto    &\cM\otimes\cL^r\end{array} $$
 defines a bijection of sets for every $d,r\in\ZZ$ and $n\geq 1$. We have to determine the orbits of $\Pic^0 X$ in $\cB_2^0(X)$ and $\cB_2^1(X)$ to verify the theorem. We already know that $\cM_0\in\cB_2^0(X)$ and $\cM_x\in\cB_2^1(X)$ for all $x\in\Cl^1 X$.
 
 For the case $d=0$, we use the following result of Atiyah.

\begin{thm}[{\cite[Thm.\ 5 (ii)]{Atiyah}}]
 For all $\cM,\cM'\in\cB_n^0(\overline X)$, there exists a unique $\cL\in\Pic^0 \overline X$ such that $\cM\simeq\cM'\otimes\cL$.
\end{thm}

 This implies that for every $\cM\in\cB_2^0(X)$, there exists a unique $\cL\in\Pic^0\overline X$ such that $\cM\simeq\cM_0\otimes\cL$.
 The action of $\Pic^0 \overline X$ and $\Gal(\rquot{\overline \FF_q}{\FF_q})$ on vector bundles over $\overline X$ commute,
 and thus for every $\sigma\in\Gal(\rquot{\overline \FF_q}{\FF_q})$,
 $$ \cM_0\otimes\cL^\sigma\ \simeq\ (\cM_0\otimes\cL)^\sigma\ \simeq\ \cM^\sigma\ \simeq\ \cM \ \simeq\ \cM_0\otimes\cL \;. $$
 By uniqueness, $\cL^\sigma\simeq\cL$; thus $\cL\in\Pic^0 X$. Hence $[\cM]=s_0\in\PBungi X$.

 For $d=1$, we restate Atiyah's classification of indecomposable vector bundles over $\overline X$. 
 
\begin{thm}[{\cite[Thm.\ 7]{Atiyah}}]
 \label{Atiyah_thm7}
 There are bijections $\varphi_n^d: \cB_n^d(\overline X) \to\Pic^0(\overline X)$
 such that the diagrams
 $$ \xymatrix{{\cB_n^d(\overline X)}\ar[r]^{\varphi_n^d} \ar[d]^{\det} & {\Pic^0(\overline X)}\ar[d]^{(n,d)} \\
             {\cB_1^d(\overline X)}\ar[r]^{\varphi_1^d}               & {\Pic^0(\overline X)} } $$
 commute for all $d\in\ZZ$ and $n\geq 1$. Here $(n,d)$ denotes multiplication with the greatest common divisor of $n$ and $d$.
\end{thm}
 
 This means that $\det: \cB_2^1(\overline X) \to \cB_1^1(\overline X)$ is a bijection, and consequently
 the restriction \mbox{$\det: \cB_2^1(X) \to \cB_1^1(X)$} is still injective. Because every element of $\cB_1^1(X)$ is of the form $\cL_x$
 for some place $x$ of degree $1$ and because $\ \det(\cM_x)\simeq\cL_x\in\cB_1^1(X)$, we obtain that $\cB_2^1(X)=\{ \cM_x | x\in\Cl^1 X\}$.
 
 By the injectivity of the determinant map, $\cM_x\simeq\cM_y\otimes\cL$ for some $\cL\in\Pic^0 X$ 
 if and only if $\det\cM_x\simeq\det(\cM_y\otimes\cL)\simeq(\det\cM_y)\otimes\cL^2$, or, equivalently, $(x-y)\in 2\Cl^0 X$.
 This proves Theorem \ref{ell_PBungi}.
\end{proof}

\begin{rem}
 \label{const_ext_pbungi}
 Theorem \ref{ell_PBungi} shows that non-isomorphic $\PP^1$-bundles that are geometrically indecomposable may become isomorphic after extension of the base field---in contrast to the opposite result for $\PBundec X$ and $\PBuntr X$ (\cite[Lemma 6.5]{Lorscheid2}). Namely, if $x-y\notin 2\Cl^0 X$, then $s_x$ and $s_y$ are not isomorphic. But there is a finite constant extension $Y\to X$ such that $x-y\in 2\Cl^0 Y$, since geometrically the class group of an elliptic curve is divisible. Thus $s_x$ and $s_y$ become isomorphic over $Y$. For a concrete example, consider $X=X_6$, and $Y=X_6'$ as in paragraph \ref{example_X_6}.
\end{rem}

\begin{cor}
 \label{repr_s_x}
 If a rank $2$ bundle $\cM$ has $\delta(\cM)=-1$ and $\det\cM\simeq\cL_x$, then $\cM$ represents $s_x$.
\end{cor} 

\begin{proof}
 A rank $2$ bundle $\cM$ with $\delta(\cM)=-1$ must be geometrically indecomposable.
 The corollary follows from the fact that every element of $\cB_2^1(X)$ is characterised by its determinant.
\end{proof}

\begin{cor}
 Let $x\in X(\FF_q)$. Then the nucleus $\cN_x$ of the graph $\cG_x$ consists of the vertices
 $$ \Vertex\cN_x \ = \ \{\,t_D\,\}_{D\in\Cl X'-\Cl X}\ \amalg\ \{\,s_x\,\}_{x\in\Cl^1 X}
                     \ \amalg \ \{\,s_0\,\} \ \amalg \ \{\,c_D\,\}_{D\in\Cl^0 X\cup\Cl^1 X} \;. $$
\end{cor}

\begin{proof}
 Theorem \ref{ell_PBungi} describes $\PBungi X$, which is contained in the vertex set of the nucleus. The description of all other vertices follows from the definition of the nucleus and the classification of $\PBundec X$ and $\PBuntr X$ as described in paragraph \ref{char_PBundec}.
\end{proof}

%%%%%%%%%%%%%%%%%%%%%%%%%%%%%%%%%%%%%%%%%%%%%%%%%%%%%%%%%%%%%%%%%%%%%%%%%%%%%%%%%%%%%%%%%%%%%%%%%%%%%

\section{Edges}
\label{section_edges}

\noindent
In this section, we investigate the edges of the graphs $\cG_x$ for degree $1$ places $x$, i.e.\ $x\in X(\FF_q)$. In Section \ref{section_graphs}, we described these graphs up to the nucleus $\cN_x$ and in the previous section we determined the vertices of $\cN_x$. We illustrate our knowledge in Figure \ref{figure_genstr_ell}. 
 
\begin{figure}[htb] 
 \begin{center}
  \includegraphics{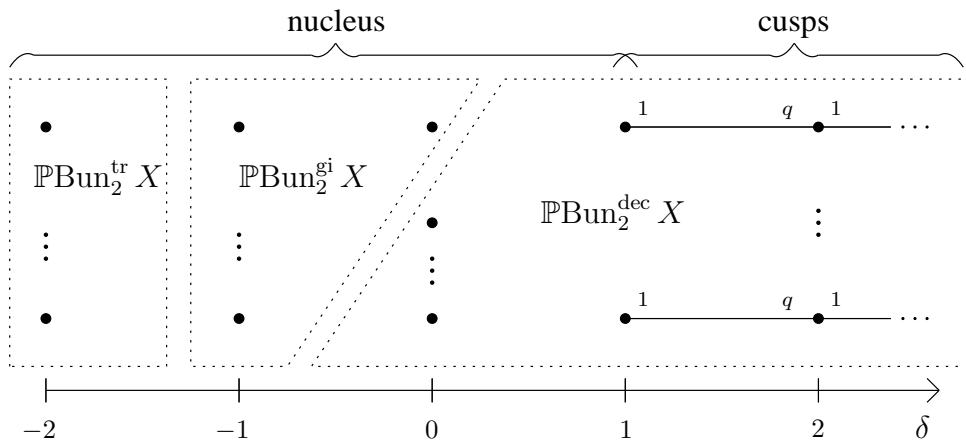}
  \caption{$\cG_x$ up to a finite number of edges} 
  \label{figure_genstr_ell}
 \end{center} 
\end{figure}

\noindent
 Fix a place $x$ of degree $1$. The divisor classes of degree $0$ can be represented by $x-z$ where $z$ is place of degree $1$ which is uniquely determined by the divisor class $x-z$. The following theorem characterizes all missing edges of $\cG_x$.

\begin{thm}
 \label{str_thm_graphs_for_genus_1}
 Let $x$ be a prime divisor of degree $1$ and $h_2=\#\Cl^0 X[2]$ the cardinality of the $2$-torsion of the class group.
 Then the edges with origin in $\cN_x$ are given by the following list:
 \begin{align*}
  \cU_x(c_0) & = \{(c_0,c_{x},q+1)\} \;, \\
  \cU_x(c_x) & = \{(c_x,c_{2x},1),(c_x,c_0,1),(c_x,s_0,q-1)\} \;, \\
  \cU_x(c_y) & = \{(c_y,c_{y+x},1),(c_y,c_{y-x},q)\} \hspace{0,6cm} \text{if }y\neq x, \\
  \cU_x(c_{y-x}) & =\{(c_{y-x},c_y,2),(c_{y-x},s_y,q-1)\} \hspace{0,6cm} \text{if }y\neq x,\text{ but }y-x\in(\Cl X)[2], \\
  \cU_x(c_{y-x}) & =\{(c_{y-x},c_y,1),(c_{y-x},c_{2x-y},1),(c_{y-x},s_y,q-1)\} \hspace{0,6cm} \text{if }y-x\notin(\Cl X)[2], \\
  \cU_x(s_0) & =\{(s_0,c_x,1),(s_0,s_x,q)\}\;, \\
  \cU_x(t_D) & =\{(t_D,s_{x+D+\sigma D},q+1)\} \hspace{0,6cm} \text{for }D\in\Cl^0 X'-\Cl^0 X \text{ and} \\
  \cU_x(s_y) & = \{(s_y,s_0,h_2)\mid \text{if } y\equiv x\pmod{2\Cl^0X}\} \vspace{0,3cm}\\
                    & \quad\cup \ \left\{\ (s_y,c_{z-x},\tfrac 12 h_2)\ \left|\  
				             \begin{subarray}{l}\text{if }(z-x)\in(\Cl^0 X)[2],\\ 
								                 z\neq x\text{ and }(z-y)\in2\Cl^0X \end{subarray}\ 
							    \right.\right\} \vspace{0,3cm}\\
                    & \quad\cup \ \left\{\ (s_y,c_{z-x},h_2)\ \left|\ 
				             \begin{subarray}{l}\text{if }(z-x)\notin(\Cl^0 X)[2]\\ 
								                 \text{and }(z-y)\in2\Cl^0X\end{subarray}\ 
							    \right.\right\} \vspace{0,3cm}\\
                    & \quad\cup \ \left\{\ (s_y,t_D,\tfrac 12 h_2)\ \left|\ 
				             \begin{subarray}{l}\text{if }D\in(\Cl^0X'-\Cl^0X),\,2D\in\Cl^0X\\ 
								                 \text{and }y\equiv D+\sigma D+x \pmod{2\Cl^0X} \end{subarray}\ 
							    \right.\right\} \vspace{0,3cm}\\
                    & \quad\cup \ \left\{\ (s_y,t_D,h_2)\ \left|\ 
				             \begin{subarray}{l}\text{if }D\in(\Cl^0X'-\Cl^0X),\,2D\notin\Cl^0X\\ 
								                 \text{and }y\equiv D+\sigma D+x \pmod{2\Cl^0X} \end{subarray}\ 
							    \right.\right\} \hspace{0,6cm}\text{for }y\in\Cl^1X.
 \end{align*}
\end{thm}

\begin{proof}
 The rest of this section is dedicated to the proof of the theorem. There are illustrations of the sets described in the theorem at the appropriate places in the proof. We draw vertices $v$ from left to right to indicate an increasing value of $\delta(v)$. In section \ref{section_examples} we show illustrations of entire graphs.

 We recall some results that we will use in the proof without further reference. The weights of all $\Phi_x$-neighbours of each vertex sum up to $q+1$ (\cite[Prop.\ 2.3]{Lorscheid2}). If $v$ and $w$ are $\Phi_x$-neighbours, then $\delta(w)=\delta(v)\pm 1$ (\cite[Lemma 8.2]{Lorscheid2}). The $\Phi_x$-neighbours $v'$ of a vertex $v=[\cM]$ with $\delta(v')=\delta(v)+1$ counted with multiplicity are in bijection with the maximal subbundles of $\cM$ (\cite[Lemma 8.4]{Lorscheid2}). This bijection is given by associating to a maximal subbundle $\cL\to\cM$ the unique sequence $ 0\to\cM'\to\cM\to\cK_x\to0 $ such that $\cL\to\cM$ lifts to a subbundle $\cL\to\cM'$ (cf.\ \cite[par.\ 8.3]{Lorscheid2}). We call this sequence the \emph{sequence associated to $\cL\to\cM$}. 

 We shall also need the following lemma. Let $\cI_x$ be the kernel of $\cO_X\to\cK_x$, which is the inverse of the line bundle $\cL_x$ in $\Pic X$. 

\begin{lemma}
 \label{lift_split}
 Let $\cL\to\cM$ be a line subbundle and 
 $$ \ses{\cM'}{\cM}{\cK_x} $$
 the associated sequence. Let $\cL'=\cM/\cL$. If $\cM\simeq\cL\oplus\cL'$, then $\cM'\simeq\cL\oplus\cL'\cI_x$.
\end{lemma}

\begin{proof}
 Note that $\cM'/\cL\simeq(\det\cM)\cI_x\cL^\vee\simeq\cL'\cI_x$.
 The hypothesis can be illustrated by the diagram
 $$ \xymatrix{0\ar[r] & {\cL}\ar[r]\ar@{=}[d] & {\cM'}\ar[r]\ar[d] & {\cL'\cI_x}\ar[r]\ar[dl]\arincl[d] & 0 \\
             0\ar[r] & {\cL}\ar[r] & {\cM}\ar[r] & {\cL'}\ar[r]\ar@/_/[l] & 0 \;.\hspace{-5pt} } $$
 Since the composition $\cL'\cI_x\to\cM\to\cK_x$ is zero, $\cL'\cI_x\to\cM$ lifts to $\cL'\cI_x\to\cM'$, 
 and the upper sequence also splits.
\end{proof}

 We prove the theorem case by case.

\begin{list}{$\bullet$}
            {\setlength{\leftmargin}{0cm}\setlength{\itemindent}{1em}\setlength{\listparindent}{1.5em}\setlength{\parsep}{0.0em}}
 \item 
 By \cite[Thm.\ 8.5]{Lorscheid2}, $c_x$ is the only $\Phi_x$-neighbour of $c_0$, which describes $\cU_x(c_0)$ completely:
\begin{center} \includegraphics{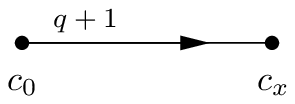} \end{center}
 
 \item 
 Let $\cM=\cL_x\oplus\cO_X$ represent $c_x$. We know from \cite[Thm.\ 8.5]{Lorscheid2} that $c_{2x}$ is the only neighbour $\cM'$ with $\delta(\cM')=2$. It has multiplicity $1$ and is given by the sequence associated to $\cL_x\to\cM$. By Lemma \ref{lift_split}, the sequence associated to $\cO_X\to\cM$ gives $\cO_X\oplus\cO_X$ as neighbour. For all other $q-1$ neighbours $\cM'$, neither $\cL_x\to\cM$ nor $\cO_X\to\cM$ lifts to $\cM'$, but then $\cL_x\cI_x\subset\cL_x\to\cM$ lifts to a subbundle $\cO_X\simeq\cL_x\cI_x\to\cM'$. We have that $\det\cM'\simeq(\det\cM)\cI_x\simeq\cO_X$, but $\cO_X\to\cM'$ cannot have a complement, since otherwise $\cO_X\to\cM$ would lift. Thus $\cM'$ must represent $s_0$. This describes $\cU_x(c_x)$:
\begin{center} \includegraphics{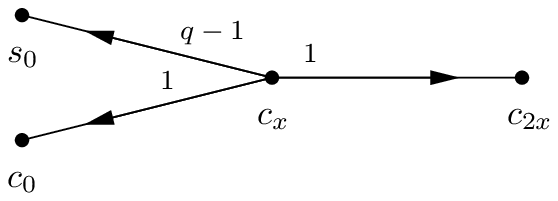} \end{center}

 \item 
 Let $\cM=\cL_y\oplus\cO_X$ represent $c_y$ with $y\neq x$. Again, we know from \cite[Thm.\ 8.5]{Lorscheid2} that $c_{y+x}$ is the only neighbour $\cM'$ with $\delta(\cM')=2$, and it has multiplicity $1$. For all other $q$ neighbours, $\cL_y\cI_x\to\cM'$ is a subbundle, and $\rquot{\cM'}{\cL_y\cI_x}\simeq\cO_X$. But since $\cL_y\cI_x\nsimeq\cO_X$, we have that $\Ext^1(\cL_y\cI_x,\cO_X)=0$ (paragraph \ref{rr_for_ell_curves}), and thus $\cM'$ decomposes. We obtain for $\cU_x(c_y)$:
\begin{center} \includegraphics{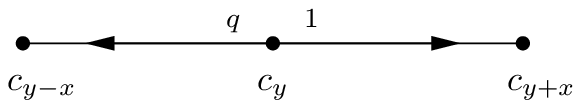} \end{center}

 \item 
 Let $\cM=\cL_y\oplus\cL_x$ represent $c_{y-x}$ with $y\neq x$. 
 Then the sequences associated to the two maximal subbundles $\cL_y\to\cM$ and $\cL_x\to\cM$ determine two neighbours $\cL_y\oplus\cO_X$ and
 $\cL_y\cI_x\oplus\cL_x$. They both decompose by Lemma \ref{lift_split} 
 and represent $c_y$ and $c_{2x-y}$, respectively. For all other $q-1$ neighbours $\cM'$, no maximal line bundle lifts, and
 thus $\delta(\cM')=-1$. Since $\det\cM'\simeq\cL_y\cL_x\cI_x\simeq\cL_y$, by Corollary \ref{repr_s_x}, $\cM'$ represents $s_y$.
 We have $c_{2x-y}=c_{y}$ if and only if $\cL_x^2\cL_y^{-1}\simeq\cL_y$, or equivalently, $\bigl(\cL_x\cL_y^{-1}\bigr)^2\simeq\cO_X$. This
 means that these two neighbours are the same if and only if $x-y\in(\Cl X)[2]$. If this is the case, we get for $\cU_x(c_{y-x})$:
\begin{center} \includegraphics{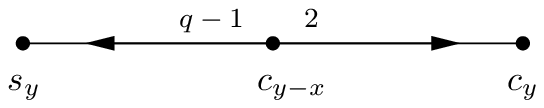} \end{center}

 \item 
 If $x-y\notin(\Cl X)[2]$, $\cU_x(c_{y-x})$ looks like:
\begin{center} \hspace{21pt}\includegraphics{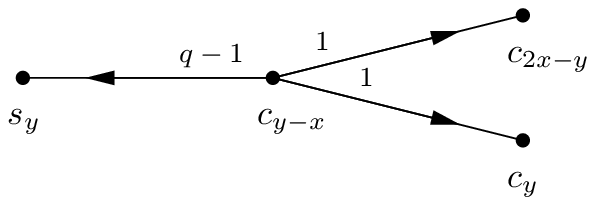} \end{center}

 \item 
 Let $\cM$ be the bundle $\cM_0$ of paragraph \ref{rr_for_ell_curves}, which represents $s_0$. Then it has a unique maximal subbundle 
 $\cO_X\to\cM$ and an associated neighbour $\cM'$ with $\delta(\cM')=1$, which decomposes. Because its maximal subbundle is $\cO_X\to\cM'$,
 $\det\cM'\simeq\cI_x$, and we recognize it as $\cO_X\oplus\cI_x$. Thus $\cM'$  represents $c_x$. 
 All $q$ other neighbours $\cM'$ have $\delta(\cM')=-1=\delta(\cM'\otimes\cL_x)$, and $\det(\cM'\otimes\cL_x)\simeq\cI_x\cL_x^2\simeq\cL_x$.
 By Corollary \ref{repr_s_x}, $\cM'\otimes\cL_x$ and thus also $\cM'$ represent $s_x$, and $\cU_x(s_0)$ is as follows:
\begin{center} \includegraphics{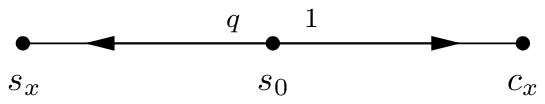} \end{center}

 \item 
 Let $\cM$ represent $t_D$ for a $D\in\Cl^0 X' - \Cl^0 X$. Since $\delta(t_D)=-2$, 
 every neighbour $\cM'$ of $\cM$ must have $\delta(\cM')=-1$.
 It is determined by its determinant, which we can calculate by extending constants to $X'$. 
 We have $\det\cM'\simeq\cI_x\det(\cL_D\oplus\cL_{\sigma D})\simeq\cI_x\cL_D\cL_{\sigma D}$. 
 Because $-x+D+\sigma D\equiv x+D+\sigma D\pmod{2\Cl X}$, Corollary \ref{repr_s_x} implies that $\cM'$ represents $s_{x+D+\sigma D}$.
 We obtain for $\cU_x(t_D)$:
\begin{center} \includegraphics{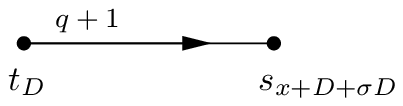} \end{center}
 
 \item 
 The most subtle part is to determine the neighbours of $s_y$ for $y\in\Cl^1 X$. We choose $\cM_y$ as representative for $s_y$, see paragraph \ref{rr_for_ell_curves}, and recall that it was defined by a nontrivial element in $\Ext^1(\cO_X,\cL_y)$. Thus $\det(\cM_y)=\cL_y$, and $\delta(\cM_y)=-1$. Look at an exact sequence
 $$ \ses{\cM'}{\cM_y}{\cK_x} \;. $$
 Then $\det(\cM')\simeq(\det\cM_y)\cI_x\simeq\cL_{y-x}\in\Pic^0 X$, and $\delta(\cM')\in\{-2,0\}$. By the symmetry of edges (see paragraph \ref{def_hecke_graphs}), $s_y$ must also be a neighbour of $[\cM']$. But we have already determined the neighbours of vertices $v$ with these properties. We find that for $(z-x)\in\Cl^0 X-\{0\}$, $c_{z-x}$ is a neighbour of $s_y$ if and only if $y\equiv z \pmod{2\Cl^0 X}$, $t_D$ with $D\in\Cl^0 X'-\Cl^0 X$ is a neighbour of $s_y$ if and only if $y\equiv x+D+\sigma D\pmod{2\Cl^0 X}$, and $s_0$ is a neighbour of $s_y$ if and only if $y\equiv x\pmod{2\Cl^0 X}$, but $c_0$ is never a neighbour of $s_y$. This shows that the theorem lists precisely the neighbours of $s_y$. There is still some work to be done to determine the weights. We begin with an observation.

\begin{lemma}
 \label{only_one_sequence}
 Up to isomorphism with fixed $\cM_y$, there is at most one exact sequence $\ 0\to\cM'\to\cM_y\to\cK_x\to0\ $ for a fixed $\cM'$.
\end{lemma}

\begin{proof}
 Suppose there are two. We derive a contradiction as follows. 
 If $\delta(\cM')\neq0$, then $\cM'$ must be a trace of a line bundle $\cL$ defined over $X'$. By extending
 constants to $\FF_{q^2}$, we may assume that $\delta(\cM')=0$ and that there are $\cL,\cL'\in\Pic^0X$ such that $\cM'$ is an extension
 of $\cL'$ by $\cL$. The composition $\cL\to\cM'\to\cM$ defines a maximal subbundle of $\cM$ because $\delta(\cL,\cM)=-1$. We get back
 the inclusion $\cM'\to\cM$ by taking the associated sequence. 
 Since we assume we have two different inclusions of $\cM'$ into $\cM$, we get
 two different subbundles of the form $\cL\to\cM$, thus an inclusion $\cL\oplus\cL\to\cM$. The cokernel is a torsion sheaf of degree $1$
 defined over $\FF_{q^2}$, say $\cK_{x'}$ for a place $x'$ of $\FF_{q^2}F$, and we obtain an exact sequence
 $$ \ses{\cL\oplus\cL}{\cM_y}{\cK_{x'}} \;; $$
 $c_0=[\cL\oplus\cL]$ is thus an $\Phi_{x'}$-neighbour of $\cM_y$. This is a contradiction as $s_y$ is not a neighbour of $c_0$.
\end{proof}

 We consider a second neighbour $\cM''$ of $\cM_y$ that represents the same element as $\cM'$ in $\PBun X$, i.e.\ 
 $\cM''\simeq\cM'\otimes\cL_0$ for some $\cL_0\in\Pic X$. Since they have the same determinant,
 $ \cL_0^2 \simeq \det(\cM'\otimes\cL_0)(\det\cM')^{-1} \simeq (\det\cM'')(\det\cM')^{-1} \simeq \cO_x $,
 we have $\cL_0\in(\Pic X)[2]$. On the other hand, Theorem \ref{Atiyah_thm7} tells us that for $\cM_y\in\cB_2^1(X)$,
 $\cM_y\otimes\cL_0\simeq\cM_y$ if and only if $\cL_0\in(\Pic X)[2]$. Thus $(\Pic X)[2]$ acts on the sequences that we investigate.
 By Lemma \ref{only_one_sequence}, we find that the multiplicity of a neighbour $\cM'$ of $\cM_y$ equals the number of isomorphism classes
 that $\cM'\otimes\cL_0$ meets as $\cL_0$ varies through $(\Pic X)[2]=(\Pic^0 X)[2]$.
 
 We begin with the case of a neighbour $\cM'$ that is associated to a maximal subbundle $\cL\to\cM_y$. Then $\delta(\cL,\cM')=0$.
 If $\rquot{\cM'}{\cL}\simeq\cL$, the only possibility with these properties is $s_0$.
 But then $\cL\to\cM'$ is the only maximal subbundle, so all $\cL\otimes\cL_0$ with $\cL_0\in(\Pic^0 X)[2]$ have different associated
 sequences, and the multiplicity of $s_0$ is therefore $h_2=\#(\Pic^0 X)[2]$.

 If $\cL':=\rquot{\cM'}{\cL}\nsimeq\cL$, 
 then $\cM'$ represents $c_{z-x}$ for the divisor $(z-x)\in\Cl^0 X$ that satisfies $\cL_{z-x}\simeq\cL'\cL^{-1}$. 
 Since $\cL_{y-x}\simeq\det\cM'\simeq\cL\cL'$, we have $z\equiv y \pmod{2\Cl^0 X}$.
 The rank $2$ bundle $\cM'$ has two different maximal subbundles, and it could happen that
 $\cM'\simeq\cM'\otimes\cL_0$ for some $\cL_0\in(\Pic^0 X)[2]-\{\cO_X\}$. This only happens if
 $\cL'\simeq\cL\cL_0$, so $\cL'\cL^{-1}\in(\Pic^0 X)[2]$, or equivalently, $(z-x)\in(\Cl^0 X)[2]$.
 Thus the multiplicity of $c_{z-x}$ as a neighbour of $s_y$ is $h_2/2$ 
 if $(z-x)\in(\Cl^0 X)[2]-\{0\}$ and $h_2$ if $(z-x)\notin(\Cl^0 X)[2]$.
 
 The last case is that of $\delta(\cM')=-2$, where $\cM'$ is the trace of a line bundle $\cL_D$, where $D\in\Cl^0 X'-\Cl^0 X$.
 If we lift the situation to $X'$, then $\cM'\simeq\cL_D\oplus\cL_{\sigma D}$, 
 and we see as in the preceding case that $\cM'\simeq\cM'\otimes\cL_0$ for some $\cL_0\in(\Pic^0 X)[2]-\{\cO_X\}$
 if and only if $D-\sigma D\in(\Cl^0 X)[2]$. This is equivalent to the two conditions $D-\sigma D\in\Cl^0 X$ and \mbox{$2D-2\sigma D=0$}, 
 or $2D=(D-\sigma D)+(D+\sigma D) \in\Cl^0 X$ and $2D = \sigma(2D)$,
 respectively, both saying that $2D\in\Cl^0 X$. This finally gives for $D\in\Cl^0 X'-\Cl^0 X$ that 
 $t_D$ has multiplicity $h_2/2$ as neighbour of $s_y$ if $2D\in\Cl^0 X$ and $h_2$ if $2D\notin\Cl^0 X$. We illustrate this below.
 The dashed arrow only occurs if $y-x\in2\Cl^0 X$. The indices $z$ and $D$ take all possible values as in the theorem,
 and $\alpha\in\{1/2,1\}$ depends on the particular edge.
 \begin{center} \includegraphics{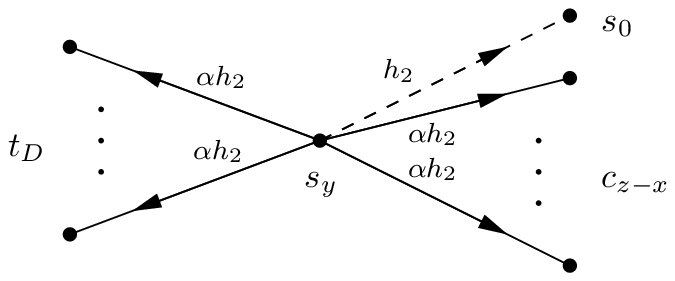} \end{center}
\end{list}
 This completes the proof of the theorem. 
\end{proof}

%%%%%%%%%%%%%%%%%%%%%%%%%%%%%%%%%%%%%%%%%%%%%%%%%%%%%%%%%%%%%%%%%%%%%%%%%%%%%%%%%%%%%%%%%%%%%%%%%%%%%

\section{Examples}
\label{section_examples}

\noindent
 This section provides some examples of graphs of Hecke operators.

\begin{eg}[class number one]\label{eg_class_number_1}
 The easiest examples are given by elliptic curves with only one rational point $x$. These examples can be found in the literature, cf.\ \cite{Cornelissen-Lorscheid}, \cite[2.4.4 and Ex. 3 of 2.4]{Serre} or \cite{Takahashi}. There are up to isomorphism three such elliptic curves: $X_2$ over $\FF_2$ defined by the Weierstrass equation $\underline Y^2+\underline Y=\underline X^3+\underline X+1$, $X_3$ over $\FF_3$ defined by the Weierstrass equation $\underline Y^2=\underline X^3+2\underline X+2$ and $X_4$ over $\FF_4$ defined by the Weierstrass equation $\underline Y^2+\underline Y=\underline X^3+\alpha$ with $\FF_4=\FF_2(\alpha)$. Since the class number is $1$, $\PBundec X_q = \{c_{nx}\}_{n\geq0}$ and $\PBungi X_q=\{s_0,s_x\}$ for $q\in\{2,3,4\}$. One calculates that $\Cl^0(X_2\otimes\FF_4)\simeq\ZZ/5\ZZ$, $\Cl^0(X_3\otimes\FF_9)\simeq\ZZ/7\ZZ$ and $\Cl^0(X_4\otimes\FF_{16})\simeq\ZZ/9\ZZ$, thus $\PBuntr X_q$ has $q$ different elements $t_1,\dotsc,t_q$ for $q\in\{2,3,4\}$. We obtain Figure \ref{ell_h=1_graph}.
\end{eg}

\begin{figure}[htb] 
 \begin{center} 
  \includegraphics{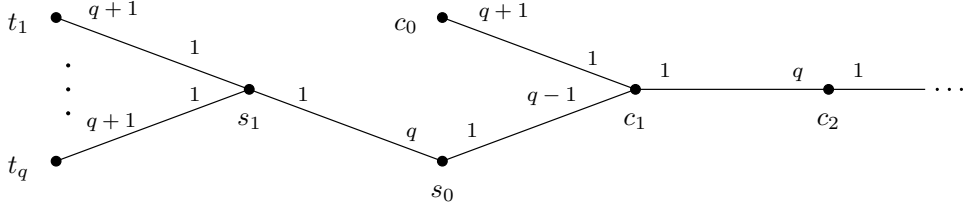} 
  \caption{$\cG_x$ for the unique degree one place $x$ of the elliptic curves $X_q$ for $q=2,3,4$} 
  \label{ell_h=1_graph}
 \end{center} 
\end{figure}

\begin{eg}[odd class number]
 \label{illustrations_ell_graphs}
 More generally, let the class number $h$ be odd. Let $x$ a place of degree $1$. Then $\cG_x$ has only one component. We write $\{x,z_2,\ldots,z_{h}\}=\Cl^1 X$ where the $z_i$'s are ordered such that $z_{2i}-x=x-z_{2i+1}$ for $i=1,\ldots,(h-1)/2$ and $\{t_1,\ldots,t_{r'}\}=\PBuntr X$. Then we can illustrate the graph of $\Phi_x$ as in Figure \ref{ell_h_odd_graph}.
\end{eg}

\begin{figure}[htb] 
 \begin{center} 
  \includegraphics{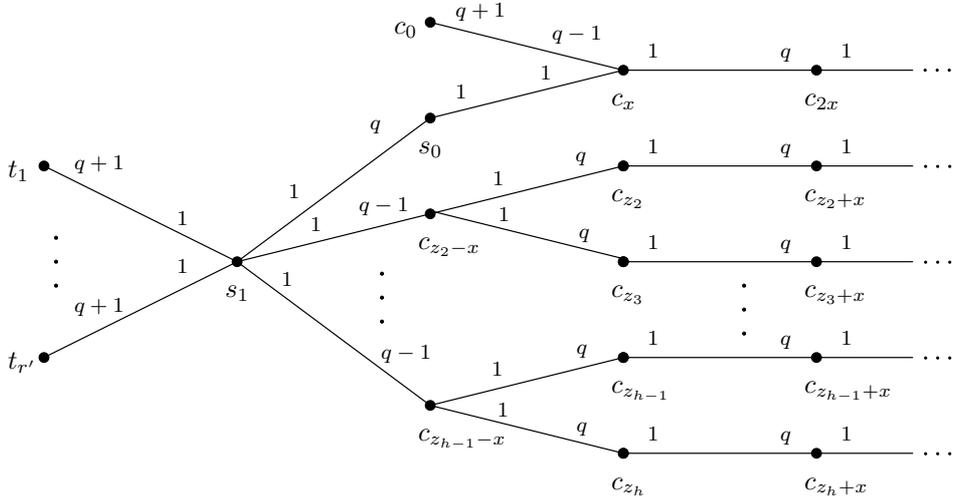} 
  \caption{$\cG_x$ for a degree one place $x$ of an elliptic curve with odd class number} 
  \label{ell_h_odd_graph}
 \end{center} 
\end{figure}

 We give two examples for elliptic curves with even class number.
 Both examples are elliptic curves over $\FF_3$ with class number $4$, but with respective class group $\ZZ/4\ZZ$ and $(\ZZ/2\ZZ)\times(\ZZ/2\ZZ)$.

\begin{eg}
 The first example is the elliptic curve $X_5$ over $\FF_3$ defined by the Weierstrass equation $\underline Y^2=\underline X^3+\underline X+2$, which has class group $\Cl^0 X_5\simeq\ZZ/4\ZZ$. Let $X(\FF_3)=\{x,y,z,z'\}$ such that $x-y$ is the element of order $2$. The number of components is $h_2=2$, and $\PBungi X$ is given by $s_0$, $s_x=s_{y}$ and $s_{z}=s_{z'}$. The class group of $X_5'=X_5\otimes\FF_9$ is $\Cl^0 X_5'\simeq(\ZZ/4\ZZ)^2$, thus $\rquot{\Cl^0 X_5'}{\Cl^0 X_5}\simeq\ZZ/4\ZZ$. Let $\{0,D,D',D''\}$ be representatives such that $D$ is the divisor with $2D\in\Cl^0 X_5$. Then $\PBuntr X_5$ contains the two elements $t_D$ and $t_{D'}=t_{D''}$. We do not need to calculate the norm map $\Cl^0 X_5'\to\Cl^0 X_5$ as we can find out to which of $t_{D}$ and $t_{D'}$ the vertices $s_x$ and $s_z$ are connected by the constraint that the weights around $s_x$ and $s_z$, respectively, sum up to $4$. The graph is illustrated in Figure \ref{ell_h=4_graphA}.
\end{eg}
 
\begin{figure}[htb] 
 \begin{center} 
  \includegraphics{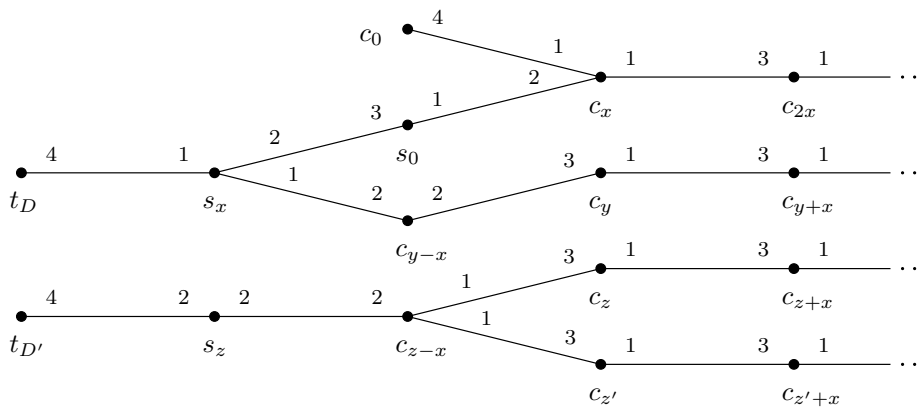} 
  \caption{$\cG_x$ for a degree one place $x$ of the elliptic curve $X_5$} 
  \label{ell_h=4_graphA}
 \end{center} 
\end{figure}

\begin{figure}[htb] 
 \begin{center} 
  \includegraphics{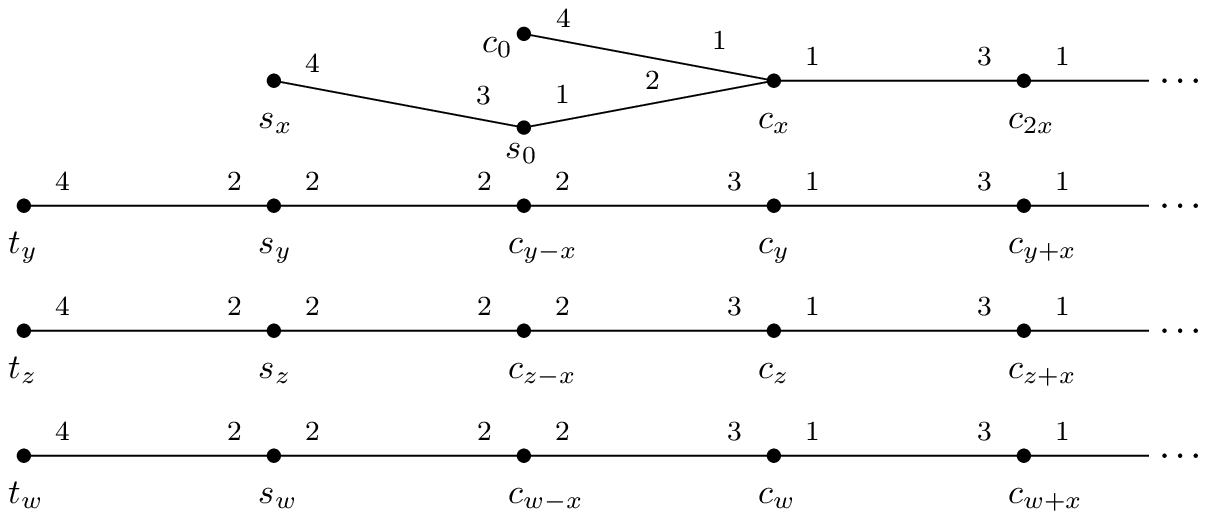} 
  \caption{$\cG_x$ for a degree one place $x$ of the elliptic curve $X_6$} 
  \label{ell_h=4_graphB}
 \end{center} 
\end{figure}

\begin{eg} 
 \label{example_X_6}
 The second example $X_6$ over $\FF_3$ is defined by the Weierstrass equation $\underline Y^2=\underline X^3+2\underline X$, and
 has class group $\Cl^0 X_6\simeq(\ZZ/2\ZZ)^2=\{x,y,z,w\}$. Here $h_2=4$, and $s_x$, $s_y$, $s_z$ and $s_w$
 are pairwise distinct vertices. For $X_6'=X_6\otimes\FF_9$, $\Cl^0 X_6'\simeq(\ZZ/4\ZZ)^2$, 
 thus $\rquot{\Cl^0 X_6'}{\Cl^0 X_6}\simeq(\ZZ/2\ZZ)^2$, which we represent by $\{0,D_1,D_2,D_3\}$, each of the $D_i$ being of order $2$. 
 Again, by the constraint that weights around each vertex sum up to $4$, we find that $\PBuntr X_6$ contains three different traces 
 of the line bundles corresponding to $D_1$, $D_2$ and $D_3$, which we denote by $t_y$, $t_z$ and $t_w$, 
 and which are connected to $s_y$, $s_z$ and $s_w$, respectively. 
 The graph is illustrated in Figure \ref{ell_h=4_graphB}.
\end{eg}

\bigskip

%%%%%%%%%%%%%%%%%%%%%%%%%%%%%%%%%%%%%%%%%%%%%%%%%%%%%%%%%%%%%%%%%%%%%%%%%%%%%%%%%%%%%%%%%%%%%%%%%%%%%
%%%%%%%%%%%%%%%%%%%%%%%%%%%%%%%%%%%%%%%%%%%%%%%%%%%%%%%%%%%%%%%%%%%%%%%%%%%%%%%%%%%%%%%%%%%%%%%%%%%%%

\part{Applications to automorphic forms}
\label{part_applications}

\bigskip

%%%%%%%%%%%%%%%%%%%%%%%%%%%%%%%%%%%%%%%%%%%%%%%%%%%%%%%%%%%%%%%%%%%%%%%%%%%%%%%%%%%%%%%%%%%%%%%%%%%%%%%%%%%%%%%%%%%%%%%%%%%%%%%%%%%%%%%%%%

\section{Automorphic forms as functions on graphs}
\label{section_automorphic}

\noindent 
In this section, we will review the notion of an automorphic form as a function on the vertex set $\PBun_2 X$ of the graphs $\cG_x$. These geometrically defined automorphic forms correspond to unramified automorphic forms for $PGL_2$ over $F$ in the more common adelic language. The geometric viewpoint lets us extract explicit eigenvalue equations from the graphs $\cG_x$, which we will list at the end of this section. This allows us to calculate the space of cusp forms and the space of toroidal automorphic forms in the following sections.

\begin{figure}[htb]
 \begin{center}
  \includegraphics{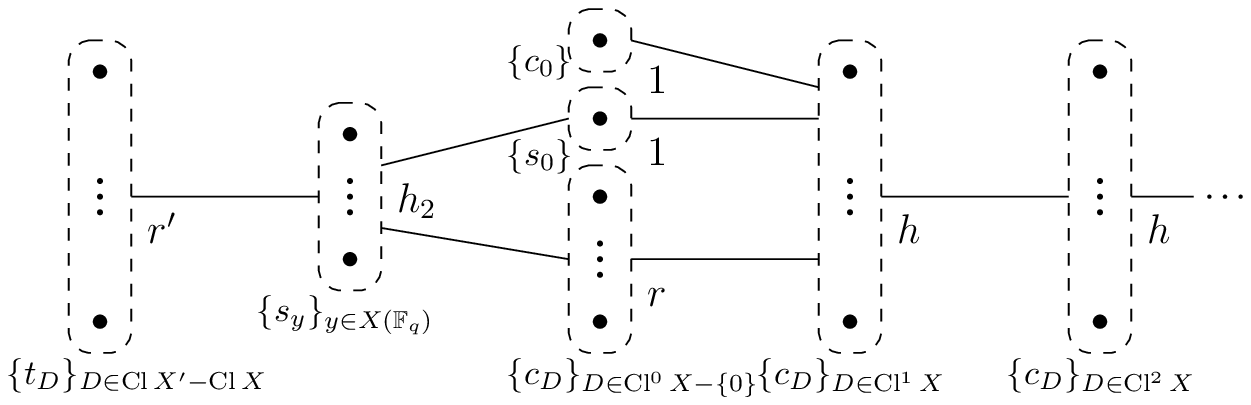}
  \caption{Certain subsets of $\Vertex\cG_x$ and their cardinalities}
  \label{clustergraph_ell}
 \end{center} 
\end{figure}

\begin{pg}
 \label{notation_toroidal_forms_for_genus_1}
 To begin with, we describe a useful picture of the graphs $\cG_x$. Define the following numbers:
 $$\begin{array}{llllll}
    h  \hspace{-7pt}&=\hspace{-7pt}&\#\Cl^0X=\#\{c_D\}_{D\in\Cl^1X},
  & h'  \hspace{-7pt}&=&\hspace{-7pt}\#(\Cl^0X'/\Cl^0X),\vspace{4pt}\\
    h_2\hspace{-7pt}&=\hspace{-7pt}&\#\Cl^0X[2]=\#\{s_y\}_{y\in X(\FF_q)},    
  & h_2'\hspace{-7pt}&=&\hspace{-7pt}\#(\Cl^0X'/\Cl^0X)[2],\vspace{4pt}\\
    r  \hspace{-7pt}&=\hspace{-7pt}&(h+h_2)/2-1=\#\{c_D\}_{D\in\Cl^0X-\{0\}},\hspace{10pt} 
  & r'  \hspace{-7pt}&=&\hspace{-7pt}(h'+h'_2)/2-1=\#\{t_D\}_{D\in\Cl X'-\Cl X}.
 \end{array}$$
 The equality in the definition of $h_2$  follows from Theorem \ref{ell_PBungi} and the equalities in definitions of $h$, $r$ and $r'$ follow from what we explained in paragraph \ref{char_PBundec}. Figure \ref{clustergraph_ell} shows certain subsets of $\Vertex\cG_x$. Each dashed subset of $\Vertex\cG_x$ is defined by the set written underneath. The integer written to the right is its cardinality. A line between two dashed areas indicates that there is at least one edge in $\cG_x$ between two vertices in the corresponding subsets.
\end{pg}

\begin{lemma}
 \label{h'=2(q+1)-h}
 $h' \ = \ 2(q+1)-h$.
\end{lemma}

\begin{proof}
 Fix a place $x$ of degree $1$ and consider $\cG_x$. We count the weights around the $h_2$ vertices $s_y$, where $y$ varies through $\Cl^1X$ modulo adding a class in $2\Cl^0X$. We know that the weights around each of the $h_2$ vertices add up to $q+1$. On the other hand, Theorem \ref{str_thm_graphs_for_genus_1} tells us precisely which vertices occur as $\Phi_x$-neighbours of the $s_y$'s and with which weight. We count all weights around the $s_y$'s:
 \begin{itemize}
  \item The vertex $s_0$ occurs with weight $h_2$. 
  \item The vertex $c_{z-x}$ occurs with weight $h_2$ if $z-x\in\Cl^0X-\{0\}$ and $z-x\neq x-z$. 
  \item The vertex $c_{z-x}$ occurs with weight $h_2/2$ if $z-x\in\Cl^0X-\{0\}$ and $z-x=x-z$. 
  \item The vertex $t_D$ occurs with weight $h_2$ if $D\in\Cl X'-\Cl X$ and $2D\notin\Cl X$.
  \item The vertex $t_D$ occurs with weight $h_2/2$ if $D\in\Cl X'-\Cl X$ and $2D\in\Cl X$.
 \end{itemize} 
 Since $c_{z-x}=c_{x-z}$, the sum of the weights of the $c_{z-x}$'s is $(h_2/2)(h-1)$. Since $t_D=t_{-D}$ and $t_D$ depends only on the class of $D$ modulo $\Cl X$, the sum of the weights of the $t_D$'s is $(h_2/2)(h'-1)$. Adding up all these contributions gives
 $$ h_2(q+1) \ = \ h_2 \, + \, (h_2/2)\,(h-1)\, +\, (h_2/2)\,(h'-1) \;, $$
 which implies the relation of the lemma. 
\end{proof}

\begin{rem}
 This result can also be obtained from the equality $\zeta_{F'}(s)=\zeta_F(s)L_F(\chi,s)$, where $\chi=\norm \ ^{\pi\i/\ln q}$ is the Hecke character corresponding to $F'$ by class field theory. These function can be written out explicitly as
 $$ \frac{q^2T^4+(hh'-q^2-1)T^2+1}{(1-T^2)(1-q^2T^2)} \ \ = \ \ \frac{qT^2+(h-q-1)T+1}{(1-T)(1-qT)} \ \cdot \ \frac{qT^2-(h-q-1)T+1}{(1+T)(1+qT)} \;, $$
 where $T=q^{-s}$. Comparing the coefficients of the numerators of these rational functions in $T$ yields an alternative proof of the lemma.
\end{rem}

\begin{pg}
 We review the notion of Hecke operators and automorphic forms in geometric language. Let $\cV$ be the space of complex valued functions on the vertex set $\PBun_2 X$ of the graphs $\cG_x$. The action of the \emph{Hecke operators $\Phi_x$} on $\cV$ is as follows. Given a function $f\in\cV$ and a vertex $v\in\PBun_2 X$, we define 
 $$ \Phi_x(f) (v) \quad = \quad \sum_{\substack{(v,v',m)\,\in\,\Edge\cG_x}} m\cdot f(v'). $$
 The actions of $\Phi_x$ and $\Phi_y$ commute for every $x,y\in\norm X$. The \emph{Hecke algebra} is the free algebra $\cH=\CC[\Phi_x]_{x\in\norm X}$ generated by the Hecke operators $\Phi_x$. It acts on $\cV$ by linear continuation of the action of the $\Phi_x$.

 A function $f\in\cV$ is an \emph{automorphic form} if $f$ is contained in a finite-dimensional $\cH$-invariant subspace of $\cV$. We denote the space of automorphic forms by $\cA$. Note that $f$ is automorphic if and only if $\{\Phi_x^i(f)\mid i\geq0\}$ is contained in a finite-dimensional subspace of $\cV$ for one $x\in\norm X$. Moreover, every $\Phi_x$-eigenspace for a fixed place $x$ has a basis that consists of simultaneous eigenfunctions for all $\Phi\in\cH_K$. We call such a simultaneous eigenfunction an \emph{$\cH$-eigenfunction}.
\end{pg}

\begin{rem}
 An automorphic form as defined here corresponds to an unramified automorphic form for $\PGL_2$ over the function field $F$ and an element $\Phi\in\cH$ corresponds to an unramified Hecke operators for $\PGL_2$. For more details on this correspondence, see section 5 in \cite{Lorscheid2}. Note that there is a digression in notation: all upper and lower indices $K$ are suppressed in this paper. For instance, we write $\cA$ for the space denoted by $\cA^K$ and $\cH$ for the algebra denoted by $\cH_K$ in \cite{Lorscheid2} and \cite{Lorscheid1}.
\end{rem}

\begin{pg}
 \label{eigenvalue_equations_in_vertices_of_the_nucleus}
 Let $f\in\cA$ be an $\cH$-eigenfunction and let $\lambda_x$ be the eigenvalue for $\Phi_x$ where $x\in\norm X$. We can evaluate the eigenvalue equation $ \Phi_x(f)=\lambda_x f $ at each vertex $v$ of $\cG_x$ and for each place $x$ of degree $1$, which yields the following equations for the vertices in the nucleus. (Note that the expressions in the column on the left are labels, which will be used for the purpose of reference.)

 \begin{align*}
  \label{x,t_D} \tag{$x,t_D$} \lambda_x f(t_D) &= (q+1)f(s_{D+\sigma D+x})                        &\hspace{-10cm}\text{for }D\in\Cl  X'-\Cl X,\hspace{0cm}\\
  \label{x,s_0} \tag{$x,s_0$} \lambda_x f(s_0) &= qf(s_x)+f(c_x) \;, \\ 
  \label{x,c_0} \tag{$x,c_0$} \lambda_x f(c_0) &= (q+1)f(c_x) \;, \\ 
  \label{x,c_z-x} \tag{$x,c_{z-x}$} \lambda_x f(c_{z-x}) &= (q-1)f(s_z)+f(c_z)+f(c_{2x-z})        &\hspace{-10cm}\text{for }z\in X(\FF_q)-\{x\},\hspace{0cm}\\
  \label{x,c_x} \tag{$x,c_x$} \lambda_x f(c_x) &= (q-1)f(s_0)+f(c_0)+f(c_{2x}) \;, \\ 
  \label{x,c_z} \tag{$x,c_z$} \lambda_x f(c_z) &= qf(c_{z-x})+f(c_{z+x})                          &\hspace{-10cm}\text{for }z\in X(\FF_q)-\{x\},\hspace{0cm}\\
  \label{x,s_y} \tag{$x,s_y$} \lambda_x f(s_y) &= \alpha f(s_0) + (h_2/2)\hspace{-0,4cm} \sum_{\substack{(z-x)\in\Cl^0 X\\ (z-x)\neq0\\ (z-y)\in2\Cl^0 X}} \hspace{-0,4cm} f(c_{z-x}) +(h_2/2)\hspace{-0,8cm}\sum_{\substack{[D]\in\Cl  X'/\Cl  X\\ [D]\neq\Cl X\\ D-\sigma D+x-y\in2\Cl^0 X}} \hspace{-0,8cm}f(t_D)\\ 
                              \hspace{4cm}\     &&\hspace{-10cm}\text{for }y\in X(\FF_q)\text{, where }\alpha=\left\{\begin{array}{ll} h_2 & \text{if }(y-x)\in2\Cl^0X,\\0&\text{if }(y-x)\notin2\Cl^0X.\end{array}\right.\!\!\!\\
 \end{align*}

 If we add up all the eigenvalue equations evaluated in the vertices $s_y$, where we let $y$ range over all of $X(\FF_q)=\Cl^1 X$, then we obtain that
 \begin{align*}
  \label{x,sum_s_y} \tag{$x,\sum s_y$} \hspace{1,5cm}\sum_{y\in X(\FF_q)}\hspace{-0,1cm}\lambda_y f(s_y)\ &=\ hf(s_0) + (h/2)\hspace{-0,4cm} \sum_{\substack{(z-x)\in\Cl^0 X\\ (z-x)\neq0}} \hspace{-0,4cm} f(c_{z-x})                         + (h/2)\hspace{-0,5cm}\sum_{\substack{[D]\in\Cl  X'/\Cl  X\\ [D]\neq\Cl  X}} \hspace{-0,5cm}f(t_D) \;.
 \end{align*}
\end{pg}

%%%%%%%%%%%%%%%%%%%%%%%%%%%%%%%%%%%%%%%%%%%%%%%%%%%%%%%%%%%%%%%%%%%%%%%%%%%%%%%%%%%%%%%%%%%%%%%%%%%%%%%%%%%%%%%%%%%%%%%%%%%%%%%%%%%%%%%%%%

\section{The space of cusp forms}
\label{section_cusp}

\noindent
 In this section, we use our knowledge about the graphs $\cG_x$ to investigate the space $\cA_0$ of cusp forms. This means that we calculate its dimension and determine the (maximal) support of a cusp form. Further we employ a Hecke operator $\Phi_y$ for a place $y$ of degree $2$ to show that a cusp form that is a $\cH$-eigenfunction does not vanish in $c_0$.

\begin{pg}
 A \emph{cusp form} is an automorphic form $f\in\cA$ that satisfies the equation 
 $$ \sum_{\cM\in\Ext^1(\cO_X,\cO_X)} \Phi(f)(\cM) \quad = \quad 0 $$
 for all $\Phi\in\cH$ (cf.\ \cite{Gaitsgory}). We denote the space of cusp forms by $\cA_0$. The space of cusp forms admits a basis of $\cH$-eigenfunctions, and is thus invariant under $\cH$. The support of a cusp form is contained in the set of vertices $v$ with $\delta(v)\leq0$ (cf.\ \cite{Harder-Li-Weisinger} or \cite[par.\ 9.3]{Lorscheid2}). 

 Let $h_2$ and $r'$ be as in the previous section. We obtain the following result.
\end{pg}

\begin{thm}\ 
 \label{dimension_and_eigenvalues_of_cusp_forms}
 \begin{enumerate}
  \item\label{cusp1} The dimension of $\cA_0$ is $r'+1-h_2$. 
  \item\label{cusp2} The support of $f\in\cA_0$ is contained in $\{t_D,s_0,c_0\}_{D\in\Cl X'-\Cl X}$.
  \item\label{cusp3} If $x$ is a place of odd degree, then $\Phi_x(f)=0$.
 \end{enumerate}
\end{thm}

\begin{proof}
 \label{cusp_forms_for_genus_1}
 Let $f$ be an $\cH$-eigenform, which means, in particular, that $f$ is not trivial, and $\lambda_x$ be the eigenvalue for $\Phi_x$ where $x\in\norm X$. We first show that $\lambda_x=0$ if $x$ is of degree $1$.
 
 Assume that $\lambda_x\neq 0$, then we conclude successively:
 \begin{itemize}
  \item $f(c_0)=0$\quad by equation \eqref{x,c_0}.
  \item $f(s_0)=0$\quad by equation \eqref{x,c_x}.
  \item $f(c_{z-x})=0$\quad for all places $z\neq x$ of degree $1$ by equation \eqref{x,c_z}.
  \item $f(s_y)=0$\quad for all places $y$ of degree $1$ by equations \eqref{x,s_0} and \eqref{x,c_z-x}.
  \item $f(t_D)=0$\quad for all $D\in\Cl X'-\Cl X$ by equation \eqref{x,t_D}.
 \end{itemize}
 Thus $f$ must be trivial, which contradicts our assumption on $f$. This shows that $\lambda_x=0$ for all places $x$ of degree $1$. We make the following successive conclusions:
 \begin{itemize}
  \item $f(s_y)=0$\quad for all places $y$ of degree $1$ by equation \eqref{x,t_D}.
  \item $f(c_{z-x})=0$\quad for all places $z\neq x$ of degree $1$ by equation \eqref{x,c_z}.
  \item $f(c_0)+(q-1)f(s_0)=0$\quad by equation \eqref{x,c_x}.
  \item $ \alpha f(s_0) + (h_2/2)\hspace{-0,7cm}\sum\limits_{\substack{[D]\in\Cl X'/\Cl X\\ [D]\neq\Cl X\\ D-\sigma D+x-y\in2\Cl^0X}}\hspace{-0,7cm}f(t_D)=0 $\quad for all places $y$ of degree $1$ by equation \eqref{x,s_y}, where $\alpha=h_2$ if $(y-x)\in2\Cl^0 X$ and $\alpha=0$ otherwise.
 \end{itemize} 
 This means that the support of $f$ is contained in $\{t_D,s_0,c_0\}_{D\in\Cl X'-\Cl X}$, which proves \eqref{cusp2}.
 
 We have $h_2+1$ linearly independent equations for $f$ as described by the last two lines of the above list. There are no further restrictions on the values of $f$ given by the eigenvalue equations since equation \eqref{x,c_0} becomes trivial. Hence the dimension of the $0$-eigenspace of $\Phi_x$ in $\cA_0$ equals 
 $$ \#\{t_D,s_0,c_0\}_{D\in\Cl X'-\Cl X} -(h_2+1) \ = \ (r'+2)-(h_2+1) \ = \ r'+1-h_2 \;. $$
 Since there are no other eigenvalues for cusp forms, $\cA_0$ equals the $0$-eigenspace, which proves \eqref{cusp1}.

 Assertion \eqref{cusp3} follows since the support of $f$ contains only vertices $v$ with $\delta(v)$ even and the parity of two $\Phi_x$-neighbours is different if $x$ is of odd degree (cf.\ \cite[Lemma 8.2]{Lorscheid2}). This implies that $\Phi_x(f) = 0$ for every place $x$ of odd degree, which proves \eqref{cusp3}. 
\end{proof}

\begin{rem}
 The dimension can be calculated by other methods, too. Once the vertices of the graphs $\cG_x$ are determined, one can use theta series to calculate the dimension, cf.\ \cite[Satz 3.3.2]{Schleich} and \cite[Thm.\ 5.1]{Harder-Li-Weisinger}.

\end{rem}

\begin{prop}
 \label{cusp_eigenfunction_does_not_vanish_in_c_0}
 If $f\in\cA_0$ is an $\cH$-eigenfunction, then $f(c_0)\neq0$.
\end{prop} 

\begin{proof}
 Let $f\in\cA_0$ be an $\cH$-eigenfunction with $\Phi_x$-eigenvalue $\lambda_x$ such that $f(c_0)=0$. We will deduce that $f$ must be the zero function, which contradicts the assumption that $f$ is an $\cH$-eigenfunction. This will prove the lemma.
 
 First we conclude from Theorem \ref{dimension_and_eigenvalues_of_cusp_forms} and equation \eqref{x,c_x} that $f(s_0)=0$. The only other vertices that are possibly contained in the support of $f$ are of the form $t_D$ for a $D\in\Cl X'-\Cl X$. We fix an arbitrary $D\in\Cl X'-\Cl X$ for the rest of the proof and show that $F(t_D)=0$.
 
 Since $X'(\FF_{q^2})=\Cl^1X'$ maps surjectively to $\Cl X'/\Cl X$, and $t_D$ only depends on the class $[D]\in\Cl X'/\Cl X$, there is a $z\in X'(\FF_{q^2})$ such that $t_D=t_z$. The covering $p:X'\to X$ maps $z$ as well as its Galois conjugate $\sigma z$ to a place $y\in\norm X$ of degree $2$. As classes in $\Cl X'$, we have $y=z+\sigma z$.

 In the following, we will investigate the graph of the Hecke operator $\Phi_y$ with the help of the graphs of the Hecke operators $\Phi_z$ and $\Phi_{\sigma z}$, which are defined over $F'$. Recall from \cite[Lemma 6.5]{Lorscheid2} that the map $p^\ast:\PBun_2X\to\PBun_2X'$ restricts to an injective map 
 $$ p^\ast: \ \PBundec X \, \amalg \, \PBuntr X \ \longhookrightarrow \ \PBundec X' \;, $$ 
 and $p^\ast$ maps $\PBungi X$ to $\PBungi X'$. We will denote the elements in $\PBundec X'$ by $c'_D$ with $D\in\Cl X'$. Then we have in particular that $c'_0=p^\ast(c_0)$, that $c'_{z+\sigma z}=p^\ast(c_y)$ and that $c'_{z-\sigma z}=p^\ast(t_z)$, and in each case, there is no other vertex in $\PBun_2 X$ that is mapped to $c'_0$, $c'_{z+\sigma z}$, and $c'_{z-\sigma z}$, respectively.

 Recall from paragraph \ref{def_hecke_graphs} that $\cK_y$ denotes the sheaf on $X$ that is supported at $y$ with stalk $\kappa_y$. If we denote by $\cK_z$ and $\cK_{\sigma z}$ the corresponding sheaves on $X'$, we have that $p^\ast\cK_y\simeq\cK_z\oplus\cK_{\sigma z}$.

 Let $\cM,\cM'\in\Bun_2 X$ fit into an exact sequence
 $$ \ses{\cM'}{\cM}{\cK_y} \;. $$
 Extension of constants is an exact functor, thus we obtain an exact sequence
 $$ \ses{p^\ast\cM'}{p^\ast\cM}{\cK_z\oplus\cK_{\sigma z}} \;, $$
 which splits into two exact sequences
 $$ 0\to\cM''\to p^\ast\cM\to\cK_z\to0 \hspace{1cm} \textrm{and} \hspace{1cm} 0\to p^\ast\cM'\to\cM''\to\cK_{\sigma z}\to0 \;, $$
 where $\cM''\in\Bun_2 X'$ is the kernel of $p^\ast\cM\to\cK_z$.
  
 In the language of graphs, this means that for every edge 
 \begin{center} \includegraphics{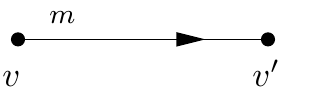} \end{center}
 between vertices $v,v'\in\PBun_2 X$ in $\Edge\cG_y$, there are a vertex $v''\in\PBun_2 X'$, and edges
 $$ \begin{array}{ccc}
     \includegraphics{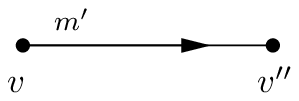}&& \includegraphics{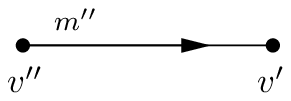} \vspace{-0,9cm}\\
     &\text{and}\vspace{0,5cm}
    \end{array} $$
 in $\Edge\cG_z$ and $\Edge\cG_{\sigma z}$, respectively.
 
 We apply this observation to find out all possibilities of $\Phi_y$-neighbours of $c_0$. The only $\Phi_z$-neighbour of $c_0$ is $c_z$, and since $z\neq \sigma z$, the $\Phi_{\sigma z}$-neighbours of $c_z$ are $c_{z-\sigma z}=p^\ast(t_z)$ and $c_{z+\sigma z}=p^\ast(c_y)$. This means that the only possible $\Phi_y$-neighbours of $c_0$ are $t_z$ and $c_y$. The vertex $c_y$ has multiplicity $q+1$ by \cite[Thm.\ 8.5]{Lorscheid2}. Thus the neighbour $t_z$ has multiplicity $(q^2+1)-(q+1)=q^2-q$. Hence $\cU_y(c_0)$ can be illustrated as
 \begin{center} \includegraphics{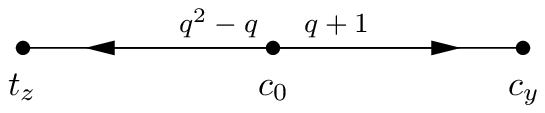} \end{center}
 
 By the assumptions on $f$, it vanishes both at $c_0$ and at $c_y$. Thus the eigenvalue equation
 $$ \lambda_f(\Phi_y)\,f(c_0) \ = \ (q+1)\,f(c_y) + (q^2-q)\,f(t_z) $$
 implies that $f(t_D)=f(t_z)=0$, which completes the proof.
\end{proof} 

%%%%%%%%%%%%%%%%%%%%%%%%%%%%%%%%%%%%%%%%%%%%%%%%%%%%%%%%%%%%%%%%%%%%%%%%%%%%%%%%%%%%%%%%%%%%%%%%%%%%%%%%%%%%%%%%%%%%%%%%%%%%%%%%%%%%%%%%%%

\section{The space of toroidal automorphic forms}
\label{section_toroidal}

\noindent
 In this section, we calculate the space of toroidal automorphic forms. We start with $F'$-toroidal automorphic forms, which have a particular nice geometric description, and prove in particular hat there are no non-trivial  $F'$-toroidal cusp forms. From this result together with the investigation of the zeros of certain $L$-functions, we shall show that the space of toroidal automorphic forms is generated by a single Eisenstein series $E(\blanc,s)$ where $s+1/2$ is a zero of the zeta function of $F$---with the exception that the characteristic of $F$ is $2$ and the class number $h$ of $F$ equals $q+1$ where the space of toroidal automorphic forms might be $2$-dimensional.

\begin{pg}
 Let $F'$ be the function field of $X'$. An automorphic form $f\in\cA$ is \emph{$F'$-toroidal} if 
 $$\sum_{\substack{[D]\in\Cl  X'/\Cl  X}} \hspace{-0,2cm} \Phi(f)(t_D) \ = \ 0 \;. $$
 for all $\Phi\in\cH$ (cf.\ \cite[Thm.\ 10.8]{Lorscheid2}). We denote the space of all $F'$-toroidal automorphic forms by $\cA_\tor(F')$. This space is a $\cH$-invariant subspace of $\cA$, and it decomposes into a direct sum of the following three $\cH$-invariant subspaces: the space $\cA_{0,\tor}(F')$ of $F'$-toroidal cusp forms, the space $\cE_\tor(F')$ that is generated by $F'$-toroidal derivatives of Eisenstein series and the space $\cR_\tor(F')$ that is generated by toroidal derivatives of residual Eisenstein series (see sections 6 and 7 in \cite{Lorscheid1} for the definition of the derivative of a (residual) Eisenstein series; we shall also give a brief description in paragraph \ref{intro_eisenstein}).

 We shall investigate these subspaces in the following. We begin with $\cA_{0,\tor}(F')$. Since $t_D=c_0$ if $D\in\Cl X$, the defining equations for $F'$-toroidality contain in particular the equation
 \begin{align*} 
  \label{T}\tag{$T$} f(c_0) \ + \ \hspace{-0,2cm} \sum_{\substack{[D]\in\Cl  X'/\Cl  X\\ [D]\neq\Cl X}} \hspace{-0,2cm} f(t_D) \ = \ 0 \;. 
 \end{align*}
 for the unit $\Phi=1$ in $\cH$.
\end{pg}

\begin{thm}
 \label{no_toroidal_cusp_forms_for_genus_1}
 Let $F'=\FF_{q^2}F$ be the constant field extension of $F$. Then the space of unramified $F'$-toroidal cusp forms is trivial.
\end{thm} 

\begin{proof} 
 Since the support of unramified cusp forms is contained in $\PBuntr X\cup\{s_0,c_0\}$ (Theorem \ref{dimension_and_eigenvalues_of_cusp_forms}), after multiplying by $2/h$, equation \eqref{x,sum_s_y} simplifies to
 $$ 0 \ = \ 2\,f(s_0)\ \ +\hspace{-0,4cm}\sum_{\substack{[D]\in\Cl  X'/\Cl  X\\ [D]\neq\Cl  X}} \hspace{-0,3cm}f(t_D) \;. $$
 Subtracting equation \eqref{T} from it yields
 $$ 0 \ = \ 2\,f(s_0)\ -\ f(c_0) \;. $$
 For cusp forms, equation \eqref{x,c_x} reads
 $$ 0 \ = \ (q-1)f(s_0)+f(c_0) $$
 and this implies that $f(c_0)=f(s_0)=0$. Thus Proposition \ref{cusp_eigenfunction_does_not_vanish_in_c_0} implies that $\cA_{0,\tor}=\{0\}$.
\end{proof} 

\begin{rem}
 A formula of Waldspurger (\cite[Prop.\ 7]{Waldspurger2}) for number fields together with certain non-vanishing results implies that a cusp form $f$ which is an $\cH$-eigenfunction is $E$-toroidal if and only if $L(\pi,1/2)\cdot L(\pi\otimes\chi_{E},1/2)=0$ where $E$ is a quadratic field extension of $F$, $\pi$ is the cuspidal representation generated by $f$ and $\chi_{E}$ is the Hecke character associated to $E$ by class field theory (cf.\ section 6 in \cite{Cornelissen-Lorscheid2}). Conjecturally the corresponding fact is also true in the function field case; see \cite{Lysenko} for partial results. 

 In our framework, $E=F'$ and $\chi_{F'}=\norm \ ^{\pi\i/\ln q}$. The case of genus $0$ follows trivially since there are no (unramified) cusp forms. For genus $1$, both $L(\pi,s)$ and $L(\pi\otimes\chi_{F'},s)$ equal $\frac{1}{(1-q^{-s})(1-q^{1-s})}$ and therefore they do not vanish in $s=1/2$. On the other hand, there are no toroidal cusp forms by the previous theorem. Thus the case of genus $1$ follows.
\end{rem}

\begin{pg}
 Before we proceed to determine the rest of $\cA_\tor(F')$, we need some facts on the zeros of the zeta function $\zeta_{F'}$ of $F'$. The zeta function 
 $$ \zeta_F(s)=\frac{qT^2\,+\,\bigl(h\,-\,(q+1)\bigr)T\,+\,1}{(1\,-\,T)\,(1\,-\,qT)} $$
 of $F$ satisfies the functional equation $\zeta_F(1-s)=\zeta_F(s)$. Here $h$ is the class number and $T=q^{-s}$. Thus the zeros $\zeta_F$ come in pairs $\{s,1-s\}$ (modulo $\frac{2\pi\i}{\ln q}\,\ZZ$). Note that this set contains only one element if $s\equiv1-s\pmod{\frac{2\pi\i}{\ln q}\,\ZZ}$. We say that $\{s,1-s\}$ is a pair of zeros of order $n$ if $s$ is a zero of order $n$ in the case that $s\nequiv 1-s\pmod{\frac{2\pi\i}{\ln q}\,\ZZ}$ respective if $n$ is a zero of order $2n$ in the case that $s\equiv 1-s\pmod{\frac{2\pi\i}{\ln q}\,\ZZ}$. Note that in the latter case $s$ is always a zero of even degree (\cite[Lemma 12.1]{Lorscheid1}).

 Since the denominator in the above equation for $\zeta_F(s)$ is a polynomial in $T=q^{-s}$ of degree $2$, we see that $\zeta_F(s)$ has only one pair of zeros (modulo $\frac{2\pi\i}{\ln q}\,\ZZ$), which is of order $1$. The same is true for the $L$-series
 $$ L(\chi_{F'},s)=\frac{qT^2\,-\,\bigl(h\,-\,(q+1)\bigr)T\,+\,1}{(1\,-\,T)\,(1\,-\,qT) }, $$
 namely its pair of zeros is $\{s+\frac{\pi\i}{\ln q},1-s-\frac{\pi\i}{\ln q}\}$ if $\{s,1-s\}$ is the pair of zeros of $\zeta_F(s)$. Here $\chi_{F'}=\norm\ ^{\pi\i/\ln q}$ is the unramified Hecke character associated to $F'$ by class field theory. However, the zeta function $\zeta_{F'}(s)=\zeta_F(s)\cdot L(\chi_{F'},s)$ of $F'$, considered as a function of $s$ modulo $\frac{2\pi\i}{\ln q}\,\ZZ$, has either two different pairs of zeros of order $1$ or one pair of zeros of order $2$.
\end{pg} 

\begin{lemma}
 \label{h=q+1_and_zeta_zeros}
 The following are equivalent.
 \begin{enumerate}
  \item\label{dz1} $\zeta_{F'}$ has a pair of zeros of order $2$.
  \item\label{dz2} $\frac 12+\frac{\pi\i}{2\ln q}$ is a zero of $\zeta_F$.
  \item\label{dz3} $h=q+1$.
 \end{enumerate}
\end{lemma}

\begin{proof}
 Let $s$ be a zero of $\zeta_F$.
 Then \eqref{dz1} holds if and only if $1-s\equiv s+\frac{\pi\i}{\ln q}\pmod{\frac{2\pi\i}{\ln q}\,\ZZ}$,
 which is equivalent to \eqref{dz2}.

 Put $s=\frac 12+\frac{\pi\i}{2\ln q}$
 and $T=q^{-s}=\i q^{-1/2}$. Then
 $$ \zeta_F(s) \ = \ \frac{q\i^2q^{-1}+\bigl(h-(q+1)\bigr)\,\i q ^{-1/2}+1}{(1-\i q^{-1/2})(1-\i q^{1/2})}
      \ = \ \bigl(h-(q+1)\bigr)\,\underbrace{\frac{\i q ^{-1/2}}{(1-\i q^{-1/2})(1-\i q^{1/2})}}\limits_{\neq0} $$
 is zero if and only if $h=q+1$, hence the equivalence of \eqref{dz2} and \eqref{dz3}.
\end{proof}

\begin{pg} \label{intro_eisenstein}
 We proceed with determining the space $\cA_\tor(F')$. We give a brief description of (residual) Eisenstein series and their derivatives. An \emph{unramified Hecke character} is a character on the divisor class group $\Cl X$. In particular, the principal characters $\norm D^s=q^{-s\deg D}$ are unramified Hecke characters. Let $\chi$ be an unramified Hecke character. Define $\lambda_x(\chi)=q_x^{1/2}\bigl(\chi^{-1}(x) + \chi(x)\bigr)$ for all $x\in\norm X$ and all unramified Hecke characters $\chi$ where we consider $x$ as a prime divisor and $q_x=q^{\deg x}$. The automorphic form $E(\blanc,\chi)$ is characterised, up to scalar multiple, by the eigenvalue equations
 $$ \Phi_x \bigl(E(\blanc,\chi)\bigr) \quad = \quad \lambda_x(\chi) \ E(\blanc,\chi) $$
 for all $x\in\norm X$ (see \cite[pars.\ 3.5 and 9.2]{Lorscheid1} for more details). If $\chi\neq\norm\ ^{\pm1}$, then $E(\blanc,\chi)$ is called a \emph{Eisenstein series} and if $\chi=\norm\ ^{\pm1}$, then $R(\blanc,\chi)=E(\blanc,\chi)$ is called a \emph{residual Eisenstein series} or a \emph{residuum of an Eisenstein series}. The functional equation for Eisenstein series is a linear relation between $E(\blanc,\chi)$ and $E(\blanc,\chi^{-1})$. Indeed, $\lambda_x(\chi)=\lambda_x(\chi^{-1})$ for all $x\in\norm X$.

 Define $\lambda_x^-(\chi)=q_x^{1/2}\bigl(\chi^{-1}(x) - \chi(x)\bigr)$ for all $x\in\norm X$ and all unramified Hecke characters $\chi$. If $\chi^2\neq1$ where $1=\norm\ ^0$ is the trivial character, then there is a unique automorphic form $E^{(1)}(\blanc,\chi)$ that satisfies the generalised eigenvalue equations
 $$ \Phi_x\bigl(E^{(1)}(\blanc,\chi)\bigr)\quad=\quad\lambda_x(\chi)\ E^{(1)}(\blanc,\chi)\ + \ \ln q_x\ \lambda_x^-(\chi) \ E(\blanc,\chi) $$
 for all $x\in\norm X$ (\cite[Lemmas 11.2, 11.3 and 11.7]{Lorscheid1}). The functions $E(\blanc,\chi)$ and $E^{(1)}(\blanc,\chi)$ span a $2$-dimensional $\cH$-invariant subspace of $\cA$ (\cite[Props.\ 11.4 and 11.8]{Lorscheid1}). If $\chi^2=1$, then there is a unique automorphic form $E^{(2)}(\blanc,\chi)$ that satisfies the generalised eigenvalue equations
 $$ \Phi_x\bigl(E^{(2)}(\blanc,\chi)\bigr)\quad=\quad\lambda_x(\chi)\ E^{(2)}(\blanc,\chi)\ + \ (\ln q_x)^2\ \lambda_x(\chi) \ E(\blanc,\chi) $$
 for all $x\in\norm X$ (\cite[Lemmas 11.2 and 11.3]{Lorscheid1}). The functions $E(\blanc,\chi)$ and $E^{(2)}(\blanc,\chi)$ span a $2$-dimensional $\cH$-invariant subspace of $\cA$ (\cite[Prop.\ 11.6]{Lorscheid1}). 

 Let $h_2$ be the number of $2$-torsion elements in the class group. Then we can describe the space of $F'$-toroidal automorphic forms as follows.
\end{pg}

\begin{thm}
 \label{space_of_F'_toroidal_forms}
 Let $s+1/2$ be a zero of $\zeta_F$ and $W$ the set of all unramified Hecke characters $\chi=\omega\norm \ ^{1/2}$ such that $\omega^2=1$, but $\omega\vert_{\Cl^0 X}\neq1$. If $h\neq q+1$, then $\cA_\tor(F')$ is generated by 
 $$ \Bigl\{\,E(\blanc,\norm \ ^s),\,E(\blanc,\norm \ ^{s+\pi\i/\ln q}),\,R(\blanc,\chi)\,\Bigr\}_{\chi\in W} $$
 and if $h=q+1$, then $\cA_\tor(F')$ is generated by 
 $$ \Bigl\{\,E(\blanc,\norm \ ^s),\,E^{(1)}(\blanc,\norm \ ^s), \,R(\blanc,\chi)\,\Bigr\}_{\chi\in W} \;. $$
 In particular, $\dim\cA_\tor(F')=2h_2$.
\end{thm}

\begin{proof}
 Since $\cA_\tor(F')=\cA_{0,\tor}(F')\oplus\cE_\tor(F')\oplus\cR_\tor(F')$, we can investigate the three summands separately. By Theorem \ref{no_toroidal_cusp_forms_for_genus_1}, we have $\cA_{0,\tor}(F')=0$. By \cite[Thm.\ 7.6]{Lorscheid1}, we have $\cR_\tor(F')=\{R(\blanc,\chi)\}_{\chi\in W}$. By definition, $\chi=\omega\norm \ ^{1/2}\in W$ if and only if $\omega$ is an unramified Hecke character that factors through $\rquot{\Cl X}{2\Cl X}$ and that is nontrivial if restricted to $\Cl^0F$. The group $\rquot{\Cl X}{2\Cl X}$ is a group of order $2h_2$, and its character group is of the same order. There are two quadratic characters such that $\omega\vert_{\Cl^0 X}$ is trivial, namely, the trivial character and $\norm \ ^{\pi\i/\ln q}$. Consequently, the cardinality of $W$ is $2h_2-2$.
 
 By \cite[Thm.\ 4.3]{Lorscheid1}, an Eisenstein series $E(\blanc,\chi)$ is $F'$-toroidal if and only if $L(\chi,s)L(\chi\chi_{F'},s)$ vanishes in $s=1/2$. Since this happens only for principal characters $\chi=\norm\ ^s$, we may reformulate this condition as follows: $E(\blanc,\norm\ ^s)$ is $F'$-toroidal if and only if $\zeta_F(1/2+s)L(\chi_{F'},1/2+s)=0$. If $h\neq q+1$, then $\zeta_F(1/2+s)L(\chi_{F'},1/2+s)$ has two pairs of zeros of order $1$ by Lemma \ref{h=q+1_and_zeta_zeros}. By \cite[Thm.\ 6.2]{Lorscheid1}, the two Eisenstein series in the theorem are $F'$-toroidal and by \cite[Thm.\ 4.3 (ii)]{Lorscheid1}, $\cE_\tor(F')$ is generated by these two linear independent functions. If $h=q+1$, then $\{1/2+\frac{\pi\i}{2\ln q},1/2-\frac{\pi\i}{2\ln q}\}$ is a pair of zeros of order $2$ by Lemma \ref{h=q+1_and_zeta_zeros}. Note that $\chi=\norm\ ^{\frac{\pi\i}{2\ln q}}$ is not of order $2$, but of order $4$. Thus by \cite[Thms.\ 4.3 and 6.2]{Lorscheid1}, $E(\blanc,\norm \ ^s)$ and $E^{(1)}(\blanc,\norm \ ^s)$ form a basis of $\cE_\tor(F')$ where $s=\frac{\pi\i}{2\ln q}$ is a zero of $\zeta_F(1/2+s)$.

 The dimension $\cA_\tor(F')$ is consequently $0+2+(2h_2-2)=2h_2$, which completes the proof.
\end{proof}

%\begin{rem}
% In case $q=p^a$ for a prime $p\neq2,3$ and an odd integer $a$, the curve $X$ is isomorphic to a supersingular elliptic curve if and only if $h=q+1$, cf.\ \cite[Thm.\ 4.1]{Waterhouse2}. For these $q$, an elliptic curve with function field $F$ is thus not supersingular if and only if the space of $F'$-toroidal automorphic forms admits a basis of $\cH$-eigenfunctions. However, for other $q$, there are supersingular elliptic curves with $h\neq q+1$.
%\end{rem}

\begin{rem}
 An alternative proof of the above theorem can be accomplished by considering eigenvalue equations, similar to the proof of Theorem \ref{dimension_and_eigenvalues_of_cusp_forms}. This is done in \cite[section 8.4]{Lorscheid-thesis}. The interesting aspect of the calculations with the eigenvalue equations is that one obtains the equality $\lambda_x^2=(q+1-h)^2$ for the $\Phi_x$-eigenvalues $\lambda_x$ of the Eisenstein series $E(\blanc,\chi)$ in $\cA_\tor(F')$ if $x$ is of degree $1$. The estimation $0<h<2q+2$ coming from the embedding of $X$ into $\PP^2$ yields that $\cA_\tor(F')$ is unitarizable. To prove that $\cA_\tor(F')$ is a tempered representation, which would imply the Riemann hypothesis for $X$ (cf.\ \cite[Thm.\ 9.4]{Lorscheid1}), one needs the estimate $q+1-2q^{1/2}\leq h\leq q+1+2q^{1/2}$. For more details, cf.\  \cite[par. 8.4.7]{Lorscheid-thesis}.

 The dependence of $\lambda_x$ on $h$ should come as no surprise, since the class number $h$ has an important influence on the shape of the graphs $\cG_x$. For a given elliptic curve with known class number, it is however possible to proof the Riemann hypothesis via these methods. Note that in this proof, we do not make use of the fact that the zeta function is a rational function. This leaves some hope that these methods can say something about zeta functions of number fields.
\end{rem}

\begin{pg}\label{explanations_on_toroidal}
 In the rest of this paper, we will determine the space of toroidal automorphic forms for an elliptic curve $X$. This will be done based on theorems of \cite{Lorscheid1} and calculations with $L$-series. 

 There is, more general, a definition of an $E$-toroidal automorphic form for every separable quadratic algebra extension of $F$, which is either a separable quadratic field extension of $F$ or isomorphic to $F\oplus F$. If $\cA_\tor(E)$ denotes the space of (unramified) $E$-toroidal automorphic forms, then the space of \emph{(unramified) toroidal automorphic forms} is the intersection $\cA_\tor=\bigcap\cA_\tor(E)$ where $E$ ranges over all separable quadratic algebra extensions of $F$. We shall not recall these definitions, which can be found in \cite[section 2]{Lorscheid1}.

 Important for the following conclusions is the connection of $E$-toroidal automorphic Eisenstein series and zeros of $L$-series. Namely, let $E$ be a separable quadratic algebra extension of $F$. Let $\chi_E$ be the Hecke character that is associated to $E$ by class field theory. In particular, $\chi_E$ is the trivial character for $E\simeq F\oplus F$ and $\chi_{F'}=\norm \ ^{\pi\i/\ln q}$ for the function field $F'$ of $X'$. Then an Eisenstein series $E(\blanc,\chi)$ is $E$-toroidal if and only if $L(\chi,1/2)L(\chi\chi_E,1/2)=0$ (cf.\ \cite[Cors. 4.4 and 5.6]{Lorscheid1}).  For zeros of higher order, also derivatives of Eisenstein series are toroidal. The precise statement can be found in \cite[Thm. 6.2]{Lorscheid1}. Let $\cE_\tor(E)$ denote the space that is generated by all $E$-toroidal derivatives of Eisenstein series.

 We can draw some first conclusions. Since $\cA_\tor\subset\cA_\tor(F')$, there are no toroidal automorphic cusp forms. By \cite[Thm.\ 7.7]{Lorscheid1}, there are no toroidal residues of Eisenstein series. So we are left to determine the space of toroidal (derivatives of) Eisenstein series, which contains at most the $2$-dimensional space $\cE_\tor(F')$ as determined in Theorem \ref{space_of_F'_toroidal_forms} and at least the Eisenstein series $E(\chi,1/2)$ for which $L(\chi,1/2)=0$, which span an $1$-dimensional subspace of $\cE_\tor(F')$. We will show in the following that $\cA_\tor$ equals this $1$-dimensional subspace when the characteristic of $F$ is odd or $h\neq q+1$. 
\end{pg}

\begin{prop}
 \label{toroidal_space_genus_1_split_torus}
 Let $s+1/2$ be a zero of $\zeta_F$. Then $\cE_\tor(F\oplus F)$ is $2$-dimensional. If $s\nequiv0 \pmod{\frac{\pi\i}{\ln q}\ZZ}$, then it is generated by 
 $$ \bigl\{\,E(\blanc,\norm \ ^s),\,E^{(1)}(\blanc,\norm \ ^s)\,\bigr\}\;, $$ 
 and if $s\equiv0 \pmod{\frac{\pi\i}{\ln q}\ZZ}$, then it is generated by 
 $$ \bigl\{\,E(\blanc,\norm \ ^s),\,E^{(2)}(\blanc,\norm \ ^s)\,\bigr\}\;. $$
\end{prop}

\begin{proof}
 The Eisenstein series $E(\blanc,\chi)$ is $F\oplus F$-toroidal if and only if $L(\chi,1/2)\cdot L(\chi,1/2)=0$. The only pair of zeros of $\bigl(L(\chi,1/2)\bigr)^2$ is $\{\norm \ ^s,\norm \ ^{-s}\}$ and it is of order $2$. Fix $\chi=\norm \ ^s$ such that $L(\chi,1/2)=\zeta_F(s+1/2)=0$. 
 
 If $s\nequiv0 \pmod{\frac{\pi\i}{\ln q}\ZZ}$, which is the case when $\norm\ ^s\neq\norm\ ^{-s}$, then $\bigl(\zeta_F(s+1/2)\bigr)^2$ vanishes in $s$ to order $2$. In this case, the space of $F\oplus F$-toroidal derivatives of Eisenstein series is spanned by $E(\blanc,\norm\ ^s)$, $E^{(1)}(\blanc,\norm\ ^s)$, $E(\blanc,\norm\ ^{-s})$ and $E^{(1)}(\blanc,\norm\ ^{-s})$ (cf.\ \cite[Thm. 6.2]{Lorscheid1}). By \cite[Thm.\ 11.10]{Lorscheid1}, $E(\blanc,\norm\ ^s)$ and $E^{(1)}(\blanc,\norm\ ^s)$ form a basis for this space.

 If $s\equiv0 \pmod{\frac{\pi\i}{\ln q}\ZZ}$, which is the case when $\norm\ ^s=\norm\ ^{-s}$, then $\bigl(\zeta_F(s+1/2)\bigr)^2$ vanishes in $s$ to order $4$. In this case, the space of $F\oplus F$-toroidal derivatives of Eisenstein series is spanned by $E(\blanc,\norm\ ^s)$, $E^{(1)}(\blanc,\norm\ ^s)$, $E^{(2)}(\blanc,\norm\ ^s)$ and $E^{(3)}(\blanc,\norm\ ^s)$ (cf.\ \cite[Thm. 6.2]{Lorscheid1}). By \cite[Thm.\ 11.10]{Lorscheid1}, $E(\blanc,\norm\ ^s)$ and $E^{(2)}(\blanc,\norm\ ^s)$ form a basis for this space. This completes the proof.
\end{proof}

%\begin{pg}
% Let $F$ be an elliptic function field with even class number. Then the class group has a nontrivial character $\chi_0$ of order $2$, which can be extended to a character $\chi$ of the divisor class group of order $2$  by Proposition \ref{lie_group_of_unramified_quasi-char}. Equivalently, $\chi$ is an unramified quasi-character of $\AA^\times$ of order $2$ that is trivial on $F^\times$ and whose kernel does not contain $\AA_0^\times$.
%
% By class field theory, there is an unramified quadratic field extension $E/F$ such that $N_{E/F}(\AA_E)=\ker\chi$. Let $\chi_E$ be the unramified Hecke character that corresponds to $E$, cf.\ paragraph \ref{def_chi_T}. Then $\chi=\chi_E$. Note that $\chi_E$ is not equal to $\norm \ ^s$ for any $s\in\CC$, hence $E$ is not the constant field extension, but a separable geometric field extension of $F$, i.e.\ the constant field of $E$ equals the constant field of $F$.
%
% Recall from paragraph \ref{class_field_theory} that for each quadratic field extension $E/F$,
% there is a quadratic character $\chi_E\in\Xi$ such that $\ker(\chi_E)=\N_{E/F}(\AA_E^\times)$.
% Note that if $F'/F$ is the constant field extension, 
% then $\chi_{F'}=\norm \ ^{\pi\i/\ln q}$ and $\AA_0^\times\subset \ker\chi_{F'}$
% (paragraph \ref{zeta_constant_field_ext}).
%\end{pg}
 
\begin{prop}
 \label{toroidal_space_genus_1_h_even}
 Assume that $F$ has a separable geometric quadratic unramified field extension $E$. Let $\chi_E$ be the corresponding Hecke character. Let $s+1/2$ be a zero of $\zeta_F$. Then $\cE_\tor(E)$ is $2$-dimensional and generated by
 $$ \bigl\{\,E(\blanc,\norm \ ^s),\,E(\blanc,\chi_E\norm \ ^s)\,\bigr\} \;. $$
\end{prop}

\begin{proof}
 As explained above, $E(\blanc,\chi)$ is $E$-toroidal if and only if $\chi$ is a zero of the product $L(\chi,1/2)L(\chi\chi_E,1/2)$. Note that a separable quadratic unramified field extension is geometric if and only if the corresponding Hecke character $\chi_E$ restricts to a non-trivial character on the class group $\Cl^0 X$. Therefore $\norm\ ^s$ cannot be the inverse of $\chi_E\norm\ ^s$. Consequently, the only (unramified) pair of zeros of $L(\chi,1/2)$ is $\{\norm \ ^s,\norm \ ^{-s}\}$ and it is of order $1$. The only pair of zeros of $L(\chi\chi_E,1/2)$ is $\{\chi_E\norm \ ^s,\chi_E^{-1}\norm \ ^{-s}\}$ and it is of order $1$. This proves the proposition.
\end{proof}

\begin{thm}
 \label{thmD}
 Let $F$ be an elliptic function field with class number $h$ and constants $\FF_q$. Let $s+1/2$ be a zero of $\zeta_F$.
 \begin{enumerate}
  \item \label{thmD1} If either the characteristic of $F$ is odd or $h\neq q+1$, then $\cA_\tor$ is $1$-dimensional and spanned by the Eisenstein series $E(\blanc,\norm \ ^s)$.
 \item\label{thmD2} If the characteristic of $F$ is $2$ and $h=q+1$, then $\cA_\tor$ is either $1$-dimensional and spanned by $E(\blanc,\norm \ ^s)$ or $2$-dimensional and spanned by $\bigr\{E(\blanc,\norm \ ^s),\,E^{(1)}(\blanc,\norm \ ^s)\bigl\}$.
 \end{enumerate}
\end{thm}

\begin{proof}
 As explained in paragraph \ref{explanations_on_toroidal}, $\cA_\tor$ contains $E(\blanc,\norm\ ^s)$ and is contained in the $2$-dimen\-sional space $\cE_\tor(F')$ generated by $E(\blanc,\norm\ ^s)$ and $E(\blanc,\norm\ ^{s+\pi\i/\ln q})$ if $h\neq q+1$, respective, $E(\blanc,\norm\ ^s)$ and $E^{(1)}(\blanc,\norm \ ^s)$ if $h=q+1$ (as described in Theorem \ref{space_of_F'_toroidal_forms}). It suffices to show that $E(\blanc,\norm\ ^{s+\pi\i/\ln q})$ respective $E^{(1)}(\blanc,\norm \ ^s)$ is not toroidal if the characteristic of $F$ is odd or $h\neq q+1$ to prove the theorem.

 If $h\neq q+1$, then $\zeta_{F'}$ has simple zeros by Lemma \ref{h=q+1_and_zeta_zeros}. Hence the intersection of $\cE_\tor(F')$ (as described in Proposition \ref{space_of_F'_toroidal_forms}) with $\cE_\tor(F\oplus F)$ (as described in Proposition \ref{toroidal_space_genus_1_split_torus}) is $1$-dimensional and spanned by $E(\blanc,\norm \ ^s)$.
 
 If the characteristic is odd, then the result follows from \cite[Thm.\ 8.2]{Lorscheid1}. We can also deduce it easily from the preceding as follows. We can assume that $h=q+1$, so $h$ is even. There is thus a separable geometric quadratic unramified field extension $E/F$ that corresponds to a non-trivial character $\chi_E$ on the class group $\Cl^0 X$. Thus the intersection of $\cA_\tor(F')$ with $\cE_\tor(E)$ (as described in Proposition \ref{toroidal_space_genus_1_h_even}) is $1$-dimensional and spanned by $E(\blanc,\norm \ ^s)$.
\end{proof}

\begin{cor}
 Let $F$ be an elliptic function field with constant field $\FF_q$ and class number $h$. If either the characteristic of $F$ is not $2$ or $h\neq q+1$, then there is for every unramified Hecke character $\chi$ and for every $s\in\CC$ a quadratic character $\omega\in\Xi$ such that $L(\chi\omega,s)\neq0$.\qed
\end{cor}

\begin{rem}\label{rem_on_class_number_1}
 The proof of the last theorem depends on many results from the theory for toroidal automorphic forms as developed in \cite{Lorscheid2} and \cite{Lorscheid1}. In the particular case that the class number is $1$, however, it is possible to deduce the theorem comparatively quickly by the method described in the introduction (cf.\ \cite{Cornelissen-Lorscheid}).
\end{rem}

\bibliographystyle{plain}

\end{document}